\documentclass[11pt]{amsart}

\usepackage{amsmath,amsfonts}
\usepackage{amsmath,amsfonts,latexsym}
\usepackage{amscd,amssymb,amsmath}
\usepackage{amsfonts}
\usepackage{amsbsy}
\usepackage{epsfig,afterpage}
\usepackage{psfrag}
\usepackage{graphicx}
\usepackage{wasysym}
\usepackage[all,cmtip]{xy}
\usepackage{subfigure,tikz}
\usepackage{multimedia}

\usepackage{color}
\usepackage{xcolor}

\newtheorem{lema}{Lemma}[section]
\newtheorem{teo}[lema]{Theorem}
\newtheorem{propo}[lema]{Proposition}
\newtheorem{rema}[lema]{Remark}
\newtheorem{coro}[lema]{Corollary}
\newtheorem{defi}[lema]{Definition}
\newtheorem{exem}[lema]{Example}
\newtheorem{conj}[lema]{Conjecture}
\newtheorem{prop}[lema]{Proposition}

\newcommand{\bea}{\begin{eqnarray*}}
\newcommand{\eea}{\end{eqnarray*}}
\newcommand{\zz}[1]{}

\newtheorem{theorem}[lema]{Theorem}
\newtheorem{proposition}[lema]{Proposition}

\newtheorem{lemma}[lema]{Lemma}

\newcommand{\bz}{\mathbb{Z}}

\newcommand{\bs}{\mathbb S}
\newcommand{\br}{\mathbb R}

\newcommand{\bc}{\mathbb C}

\newcommand{\bt}{\mathbb T}

\newcommand{\bq}{\mathbb Q}

\newcommand{\g}{\mathrm g}

\renewcommand{\AA}{\mathcal A}
\newcommand{\UU}{\mathcal U}

 \newcommand{\NN}{{\mathbb{N}}}
 \newcommand{\ZZ}{{\mathbb{Z}}}

\newcommand{\sct}{{\rm sct}}
\newcommand{\ct}{{\rm ct}}
\newcommand{\cat}{{\rm cat}}

\begin{document}

\title[Equivariant covering type]{Equivariant covering type and the number of vertices in equivariant triangulations}

\author[Dejan Govc]{Dejan Govc$^*$}

\author[Wac{\l}aw Marzantowicz]{Wac{\l}aw Marzantowicz$^{**}$}
\thanks{$^{**}$ Supported by the   Polish Committee of Scientific Research  Grant
Sheng 1 UMO-2018/30/Q/ST1/00228}
\address{$^{**}$ \;Faculty of Mathematics and Computer Science, Adam Mickiewicz University of
Pozna{\'n}, ul. Uniwersytetu Pozna{\'n}skiego 4, 61-614 Pozna{\'n}, Poland.}
\email{marzan@amu.edu.pl}

\author[Petar Pavesi\v{c}]{Petar Pave\v si\'{c}$^{***}$}
\thanks{$^*$, $^{***}$ Supported by the Slovenian Research Agency program P1-0292 and grant J1-4001.}
\address{$^*$, $^{***}$ Faculty of Mathematics and Physics, University of Ljubljana,
and the Institute of Mathematics, Physics, and Mechanics, Jadranska 19,  1000 Ljubljana, Slovenija}
\email{dejan.govc@gmail.com, petar.pavesic@fmf.uni-lj.si}


\begin{abstract}
We introduce the notion of the \emph{equivariant covering type} of a space $X$ on which a finite group 
$G$ acts, and study its properties. The equivariant covering type measures the size of $G$-equivariant
good covers of $X$ and is thus an extension  of the \emph{covering type} of a space, introduced by 
Karoubi and  Weibel. 
We show that the equivariant covering type is a $G$-homotopy invariant and describe its relation with
other $G$-invariants, like the equivariant LS-category, $G$-genus and  the multiplicative structures
of equivariant cohomology theories. We also compute the $G$-covering type of regular $G$-graphs, 
give estimates  for orientation-preserving actions on surfaces and for the projectivizations of complex
representations of $G$ and cohomology spheres. As an application, we derive  
estimates of sizes of minimal $G$-triangulations for various $G$-spaces. 
\end{abstract}

\maketitle

\section{Introduction} \label{sec:Introduction}

Given a topological space $X$ with an action of a finite group $G$, we consider $G$-equivariant 
good covers of $X$ and define the \emph{equivariant covering type}, $\ct_G(X)$ of  $X$ as the 
minimal size of such a cover for spaces that are $G$-homotopy equivalent to $X$ 
(see \ref{defi:sct & ct} for a precise definition).  The equivariant covering type extends the 
ordinary \emph{covering type}, $\ct(X)$ of $X$, which was introduced in 2016 by  Karoubi and 
Weibel \cite{K-W} and was later studied by the present authors in 
\cite{GovcMarzantowiczPavesic2019,GMP2,GovcMarzantowiczPavesic2022} and by Duan, Marzantowicz, 
Zhao in \cite{DMZ}. In spite of some formal similarities between $\ct_G(X)$ and $\ct(X)$, the 
intricacies of group actions lead to completely new phenomena that required introduction of 
new techniques from equivariant homotopy theory.

$G$-covering type can be viewed as a homotopy 
invariant measure of the complexity of the action of $G$ on $X$. It is closely related to other 
$G$-homotopy invariants like the $G$-equivariant Lusternik-Schinirelmann category and the 
$G$-genus of $X$. In this paper we develop the basic theory of the equivariant covering type. 
In particular we derive several lower bounds for $\ct_G(X)$ in terms of the equivariant LS-category, 
$G$-genus and the cup-length in various equivariant cohomology theories.

An important property of $\ct_G(X)$  is that it gives a lower bound for the number of orbits 
of vertices in any regular triangulation of $X$. In fact, classical theorems on the triangulability of 
$G$-actions on smooth manifolds, proved in the 1970's by Illman \cite{Illman, Illman2} and others
do not provide any information about the size of those triangulations. While for general polyhedra 
there exists a well-established study of minimal triangulations (mostly based on methods of 
combinatorial geometry - see surveys by Datta \cite{Datta} and Lutz \cite{Lutz}), as far as 
we know, there does not exist an analogous theory of minimal $G$-triangulations. 
Thus, as a second main contribution of this paper we initiate the study on minimal regular 
$G$-equivariant triangulations (cf. Bredon \cite[Def. III 1.2]{Bredon}). 

\

{\bf Main results.} 

- Proposition \ref{prop:estimate vertices by covering type} relating the equivariant covering type and the 
size of regular equivariant triangulations.

- Theorem \ref{teo: ct_G of graph}, in which we compute the $G$-covering type of a regular $G$-graph. 

- Theorem \ref{thm:surface minimal invariant triangulation}, in which we give upper and lower bounds for
  the $G$-covering type of an oriented closed surface $\Sigma_\g$ with respect to an 
  orientation-preserving action 
  of a finite group $G$. As a consequence we also obtain estimates of the sizes of equivariant triangulations 
  for a  regular, simplicial and orientation-preserving action of $G$ on $\Sigma_\g$.   The formula connects 
  the number of vertices of a minimal triangulation of $\Sigma^{\prime}_{\g^{\prime}}=\Sigma_\g/G$ 
  with the number of singular fibers of a branched cover $\pi: \Sigma_\g \to \Sigma^\prime_{\g^\prime}$. 

- Theorem \ref{thm:estimate by genus}, where we give a lower bound for the $G$-covering type $\ct_G(X)$
in terms of the equivariant genus $\gamma_G(X)$, which can be explicitly computed.

- Theorem \ref{thm:arithmetic}, where we give a lower estimate of $\ct_G(X)$ using the weighted
cohomological length of $h^*_G(X)$ for an arbitrary generalized  equivariant cohomology theory $h^*_G$.
As an application we obtain Theorem \ref{thm:estimate for projective space}, where we use 
equivariant $K$-theory to prove that $\ct_G(P(V)) \geq n^2$, where $P(V)$ is the projectivization 
of a complex $n$-dimensional representation $V$ of $G$.
Another interesting application is Theorem \ref{thm:ct of spheres for Z_p^k} in which we use the Borel
cohomology to estimate the equivariant covering type for actions of $G = (\bz_p)^k$ ($p$ a prime)
on an $n$-dimensional $\bz_p$-cohomology sphere $X$. If the action is free, then
$\ct_G(X)\geq \frac{(n+1)(n+2)}{2}$ for $p=2$, and $\ct_G(X)\geq \frac{(n+1)(n+3)}{8}$ for $p$ odd.
An estimate for the case where the action has fixed points is also provided.

{\bf Outline.}
In  Section \ref{sec:Equivariant LS category} we introduce the notation and recall some basic definitions
and results on the equivariant Lusternik-Schnirelmann category and the $G$-genus of $G$-spaces. 
In Section 
\ref{sec:Equivariant covering type} we discuss equivariant good covers, define the equivariant
covering type and show how it is related to the $G$-category and $G$-genus.
In Subsection \ref{$G$-covering type of one-dimensional complexes} we compute the equivariant 
covering type
of finite one-dimensional complexes (finite $G$-graphs). Section \ref{Equivariant covering type and
minimal triangulations of surfaces} is dedicated to the computation of the $G$-covering
type for orientation-preserving finite group actions on closed oriented  surfaces. In Section
\ref{Estimate of $G$-covering type by $G$-category and $G$-genus} we prove several 
lower estimates of $G$-covering type in terms of $G$-genus and $G$-category.  Finally, Section
\ref{sec:Cohomological estimates} is dedicated to estimates of the equivariant covering 
type of some classes of $G$-spaces based on the multiplicative structure of various 
generalized cohomology theories.

{\bf Thanks.}
The final form of this work  was shaped during the workshop "Some problems of Applied \& Computational
Topology" which was held 05.03.2023 - 11.03.2023 in the 
B{\c{e}}dlewo Conference Center. The authors would like to express their thanks to the authorities 
of Banach Center for giving them an opportunity to organize this meeting.

\section{Equivariant Lusternik-Schnirelman category} \label{sec:Equivariant LS category}

{\bf All groups considered in this work are finite.} Nevertheless, we will occasionally include
the finiteness assumption in the statements of the theorems to make them self-contained.  When discussing
concepts from equivariant topology and homotopy theory we will use freely the terminology and notation
from Bredon \cite{Bredon}.

The elementary (indecomposable) $G$-sets are the homogeneous sets of cosets $G/H$, where
$H\le G$ is an arbitrary subgroup. Given subgroups $H,K\le G$  and $g\in G$, the formula
$\varphi(gK):=gH$ gives a well-defined $G$-map $\varphi: G/K \to G/H$ between respective orbits if,
and only if, $gKg^{-1} \subseteq H$. As a consequence, the set of $G$-maps ${\rm Map}_G(G/K,G/H)$
can be identified with the set of invariants $(G/H)^K$.

A topological space $X$ with a continuous action of a group $G$ will be called a $G$-\emph{space}.
 Given $x\in X$ the set $Gx=\{gx\mid g\in G\}$ is the \emph{orbit} of $x$.
Every orbit $Gx$ can be  identified with the set of cosets $G/G_x$, where
$G_x=\{g\in G\mid gx=x\}$ is the \emph{stabilizer} of the action of $G$ at the point $x$.
A subset $A\subset X$ is $G$-\emph{invariant} if for every $g\in G$ we have  $gA=A$.
Clearly, the orbits are the minimal $G$-invariant subsets of $X$ and every invariant subset
of $X$ can be partitioned as a disjoint union of orbits.
More generally, for any $A\subset X$ we define the \emph{saturation}  $GA=\bigcup_{g\in G} gA$,
which is the smallest invariant subset of $X$ containing $A$.
An open cover $\mathcal{U}$ of $X$ is a \emph{$G$-cover} if $gU\in \mathcal{U}$ for every $g\in G$
and $U\in \mathcal{U}$. Every $G$-cover $\UU$ can be partitioned into orbits
$\widetilde{U}=\{gU\mid g\in G\}$ which are $G$-invariant subfamilies of $\UU$.

\zz{From now on by a $G$-cover  we understand an open $G$-invariant cover   $\mathcal{U} =
{\underset{s}{\, \cup\,}} \tilde{U} = G U_s$ split into orbits of open sets $U_s$ such that either
$g U_s = U_s$ or $gU_s \cap U_s = \emptyset$ for all $g\in G$.}

The classical Lusternik-Schnirelmann category $\cat(X)$ of a space $X$ is defined as the minimal
number of open categorical sets that are needed to cover $X$, where a subspace $U\subseteq X$ is
called \emph{categorical} if the inclusion $U\hookrightarrow X$ is homotopic to a constant map,
see Cornea, Lupton, Oprea, Tanre \cite{CLOT}. In the equivariant setting we require maps and
homotopies to be $G$-equivariant,
and so we need to replace points with $G$-orbits. Thus a $G$-invariant subset $U\subseteq X$ is
said to be $G$-\emph{categorical} if there is a $G$-equivariant homotopy between the inclusion
$U\hookrightarrow X$ and the inclusion of some orbit $Gx\hookrightarrow X$.
Note, that a $G$-categorical set that deforms to an orbit $Gx$ must have at least
$|Gx|=|G/G_x|$ components which are mapped to different points of $Gx$.
If there is an equivariant map from $Gx$ to $Gx'$, then a $G$-homotopy to $Gx$ can be extended
to a $G$-homotopy to $Gx'$. In particular, if $G$ has a fixed point $x_0$, then $Gx_0=\{x_0\}$ and
so every $G$-categorical set is a categorical set in the classical sense. This is why we normally
consider $G$-actions without fixed points, or alternatively, restrict the types of orbits
to which the invariant subsets of $X$ can be deformed.

Let $\mathcal{A} $ be some subset of the set $\{G/H\mid H\le G\}$. An open $G$-invariant
subset $U\subset X$ is said to be $G$-\emph{categorical} in $X$ with respect to  $\mathcal{A} $ if
$U\hookrightarrow X$ is $G$-homotopic to the inclusion $Gx\hookrightarrow X$ for some
$Gx\in\AA$.
For a  $G$-space $X$ we define the $G$-\emph{category} of $X$ with respect to $\AA$ as
the minimal cardinality (denoted $\cat_G^\AA(X)$) of a cover of $X$ by open sets that are
$G$-categorical with respect to $\AA$.
The most common choice for $\AA$ is the family of all non-trivial orbits $\AA = \{G/H\mid H\neq G\}$.
In that case we will write $\cat_G(X)$ instead of $\cat_G^\AA(X)$. Observe that if the action of
$G$ has fixed points, then $\cat_G(X)=\infty$, so we will often assume that the action of $G$ is
fixed point free.

The equivariant version of the Lusternik-Schnirelman category has been introduced by
Fadell \cite{Fadell} for free $G$-actions, and by Clapp and Puppe \cite{Clapp-Puppe} and
Marzantowicz \cite{Marzan} for the general case.  Standard references for $G$-category  are
Bartsch\ \cite{Bartsch} and Marzantowicz \cite{Marzan}. We will also need some specific
computations of the equivariant category for finite group actions on surfaces from
Gromadzki, Jezierski, Marzantowicz \cite{GroJezMar}.

There is another numerical invariant related to the $G$-category which is often easier
to compute. First of all, note that the join $X*Y$ of two $G$-spaces is again a $G$-space with respect to
the diagonal action defined as  $g(x,t,y):= (g x,t,g y)$. Given a set of orbits $\AA$,
we define the $G$-\emph{genus} of $X$ relative to $\AA$ (denoted $\gamma_G^\AA(X)$)
as the minimal integer $n$ for which there exist orbits $A_1,\ldots,A_n\in\AA$ and
a $G$-map $X\longrightarrow A_1*\cdots * A_n$. If no such $n$ exists we set
$\gamma^\AA_G(X)= \infty$. As before, if $\AA$ is the set of all non-trivial orbits of $G$,
then we abbreviate $\gamma^\AA_G(X)$ to $\gamma_G(X)$.

The main properties of the $G$-genus are summarized in the following propositions.

\begin{prop}\label{genus estimates category 1}\cite[Proposition 2.10]{Bartsch}\\
For any $G$-space X we have $\; \gamma^\AA_G(X)\, \leq \,cat^\AA_G(X)\,$.
\end{prop}

\begin{prop}\label{genus estimates category 2} (\cite[Proposition 2.15]{Bartsch})
\begin{itemize}\label{properties of genus}
\item[]\hspace*{-12mm}{\emph{$G$-invariance}}: If $X$ and $Y$ are $G$-homotopy equivalent
spaces, then
$\gamma^\AA_G(X)= \gamma^\AA_G(Y)$.
\item[]\hspace*{-12mm}{\emph{Normalization}}: $\gamma^\AA_G(X) = 1$ if, and only if $X$ is $G$-
contractible relative to $\AA$.
\item[]\hspace*{-12mm}{{\emph{Monotonicity }}: If there exists a $G$-map $X\to Y$, then
$\gamma^\AA_G(X) \leq \gamma^\AA_G(Y)$.}
\item[]\hspace*{-12mm}{{\emph{Subadditivity}}:
$\gamma^\AA_G(X \cup Y ) \leq  \gamma^\AA_G(X) + \gamma^\AA_G(Y)$}
\end{itemize}
\end{prop}

\begin{prop}\label{upper estimate of genus}(\cite[Proposition 2.16]{Bartsch})\\
If $X$ is a compact $G$-space, then $\gamma_G(X) < \infty$. If, in addition, there exists
an orbit type $G/H\in\AA$, such that every orbit $Gx\subset X$ can be $G$-equivariantly mapped to
$G/H$,
then $$\gamma^\AA_G(X)\leq \dim X +1.$$
The latter assumption is satisfied if there exists one minimal orbit type in $\AA$.
\end{prop}

The $G$-category  shares only the $G$-homotopy invariance, normalization and subadditivity from the above
listed properties. The estimate of Proposition \ref{upper estimate of genus} takes a much more
complicated form in the case of $G$-category (cf. \cite{Bartsch}, \cite{Colman}, \cite{Marzan}), and
the upper bound is greater than $\dim X +1$ in general.

We conclude the section with a short proof that $G$-category is invariant with respect to $G$-homotopy
equivalences, as this is not mentioned in the cited references.

\begin{defi}\label{preserving AA}
We say that a $G$-map $\phi: X\to Y$ preserves the family of orbits $\AA= \{(G/H_s)\}$ if for every
$X \supset Gx \simeq G/G_x\in \AA$, we have $\phi(Gx) \simeq G/ G_{\phi(x)} \in \AA$.
\end{defi}

Note that the condition of Definition \ref{preserving AA} is satisfied if $X^G=Y^G=\emptyset$ and
$\AA=\{G/H\mid H\neq G\}$.

\begin{prop}\label{G-homotopy invariance of category}
Let $X$, $Y$ be $G$-spaces. If a $G$-homotopy equivalence between $X$ and $Y$ can be given by
$G$-maps $\phi\colon X\to Y$ and $\psi\colon Y\to X$ that preserve $\AA$, then $\cat_G^\AA(X)=\cat_G^\AA(Y)$.
\end{prop}
\begin{proof}
Let $\phi: {X\overset{G}{\,\to\,}} Y$ and correspondingly  $\psi: Y{\overset{G}{\,\to\,}} X$ be such a $G$-
maps that  $\psi \circ \phi {\overset{G}{\,\sim\,}} {\rm id}_X$ and $\phi \circ \psi {\overset{G}{\,\sim\,}}
{\rm id}_Y$ respectively. Furthermore, let $\tilde{V}_1, \, \dots, \tilde{V}_k$, $k=\cat_G^\AA(Y)$ be a
$G$-categorical cover of $Y$, and $r_{i,t}: \tilde{U}_i \times I \to Y$, $1\leq i \leq k$, with  $  r_{i,1}:
\tilde{U}_i \to  G/G_{y_i}$,  the corresponding $G$-maps. Then the $G$-cover $\tilde{U}_i:= \phi^{-1}(V)$ of
$X$ is $G$-categorical. Indeed, the composition $ \psi \circ r_{i,t} \circ \phi : \tilde{U}_i \times I \to   X
$ gives the required $G$-deformation, because $\psi \circ r_{i,1} \circ \phi : \tilde{U}_i  \to \psi(G/
G_{y_i}) \subset  X$, with $\psi(G/G_{y_i})\in \AA$, and $\psi \circ\  r_{i,0} \circ \phi : \tilde{U}_i  \to
X$ is $G$-homotopic to $\psi\circ\phi\colon \tilde{U}_i\to X$. But the latter is $G$-homotopic to the inclusion
of $\tilde{U}_i$ in $X$, which shows $\cat_G^\AA(X)\leq \cat_G^\AA(Y)$. The proof of the converse inequality is
analogous.
\end{proof}

\section{Equivariant covering type} \label{sec:Equivariant covering type}

The covering type of a space is a homotopy invariant that was recently
introduced by Karoubi and Weibel \cite{K-W}. It is based on open covers that are \emph{good}
in the sense that all finite, non-empty intersections of elements of the cover are contractible.
Then the \emph{covering type} $\ct(X)$ of $X$ is defined as the minimal cardinality of a
good cover of a space that is homotopy equivalent to $X$. The well-known Nerve Theorem
(see Hatcher \cite[Cor. 4G.3]{Hatcher}) states that for every good cover $\UU$ of a paracompact
space $X$ the geometric realization $|N(\UU)|$ of the nerve of the cover is homotopy equivalent
to $X$ (in other words, $N(\UU)$ is a \emph{homotopy triangulation} of $X$). Thus, we can say that
the covering type of a paracompact space $X$ equals the minimal
number of vertices in a homotopy triangulation of $X$.

In order to extend the above results to the equivariant context we need some preparation.
A $G$-space $X$ is said to be $G$-\emph{contractible}, if there exists
an $x\in X$, such that the orbit $Gx\subset X$ is a $G$-deformation retract of $X$.

The proof of the following proposition is straightforward.

\begin{prop}
Let $X$ be $G$-contractible to an orbit $Gx$. Then $X$ has precisely $|Gx|$ path-components and
each of them is contractible to a point in the orbit $Gx$.
\end{prop}

Given a $G$-space $X$ we are interested in triangulations of $X$ that are compatible with the
group action. Recall that an abstract simplicial complex is determined by a set of vertices
$V$ and a set $K$ of finite non-empty subsets of $V$, called simplices, which is closed under
inclusion (i.e., a non-empty subset of a simplex in $K$ is also a simplex in $K$).
A simplicial map between two simplicial complexes is a map between the respective sets of
vertices that carries simplices into simplices.

\begin{defi}\label{defi:regular simplicial}(cf. Bredon \cite[Sec. III.1]{Bredon})
\begin{enumerate}
\item A \emph{simplicial $G$-complex} is a simplicial complex $K$ together with an action of $G$ on $K$ by simplicial maps.
\item A simplicial $G$-complex $K$ is \emph{regular} if the action of $G$ on $K$ satisfies the
following conditions:
\begin{itemize}
\item[R1)] If vertices $v$ and $gv$ belong to the same simplex in $K$, then $v=gv$.
\item[R2)] If $\langle v_0,\ldots,v_n\rangle$ is a simplex of $K$ and if for some choice
of $g_0,\ldots,g_n\in  G$ the points $g_0 v_0,\ldots , g_n v_n$ also span a simplex of $K$, then
there exist $g\in G$, such that $gv_i= g_i v_i$, for $i =0,\ldots\, n$ (in other words,
$\langle g_0 v_0,\ldots , g_n v_n\rangle=g\langle v_0,\ldots,v_n\rangle$).
\end{itemize}
\end{enumerate}
\end{defi}

We will also need a slightly stronger regularity condition, introduced by Illman.

\begin{defi}\label{defi:strict regular simplicial}(cf. Illman \cite[p. 201]{Illman})
A simplicial  $G$-complex  $K$ is  an \emph{equivariant simplicial complex} or a \emph{strictly regular $G$-complex} if
it satisfies \emph{R1)} and the following condition:
\begin{itemize}
\item[R3)] {For any $n$-simplex $\sigma$ of $ K$  the vertices $v_0, \,\dots, \, v_n$ of $K$ can be ordered in such
a way that we have $G_{v_n} \subset G_{v_{n-1}} \subset \,  \cdots\,  \subset G_{v_0}$.}
\end{itemize}
\end{defi}

By conditions R2) or R3), if two $n$-simplices in $K$ have vertices from the same set of orbits, then
they belong to an orbit of the action of $G$ on $K$. Thus, if $K$ is a regular $G$-complex, then
one can naturally build a quotient simplicial complex $K/G$ whose vertices are the orbits of
the action of $G$ on the vertices of $K$, and whose simplices are the orbits of the action of $G$
on the simplices of $K$. Clearly, the geometric realization $|K/G|$ of the quotient complex
is homeomorphic to the quotient space $|K|/G$.

The regularity conditions are quite stringent. For example, none of them hold for the $\ZZ_3$-action
that rotates the 2-simplex around its center. The following propositions shows that the situation improves
if we consider barycentric subdivisions.

\begin{prop}(\cite[Prop. III.1.1]{Bredon})
If $K$ is any simplicial G-complex, then the induced action on the barycentric subdivision $K'$
satisifies condition \emph{R1)}. Moreover, if the action of $G$ on $K$ satisfies \emph{R1)}, then the induced action
on $K'$ satisfies \emph{R2)} and \emph{R3)}.

Therefore, any simplicial action of $G$ on $K$ induces a strictly regular action,
on the second barycentric subdivision of $K$.
\end{prop}

Thus, we can always achieve regularity of a group action by increasing the number of simplices, but here
our goal is to estimate what is the minimal size of a triangulation for which a given group action is regular.
To this end we need a relation between simplicial $G$-complexes and $G$-covers. The following
notions were introduced by Yang \cite{Yang}.

\begin{defi}\label{defi:regular cover}
An open $G$-cover $\mathcal{U}$ of a $G$-space $X$ is \emph{regular} if the following conditions
hold:
\begin{itemize}
\item[RC1)] For every $U \in \UU$ and $g\in G$, either $U=gU$ or $U\cap gU=\emptyset$
\item[RC2)] If $U_0,\ldots, U_n$ are elements of $\UU$ with non-empty intersection and if for some
choice of elements $g_0,\ldots,g_n\in  G$ the intersection of sets
$g_0U_0,\ldots, g_nU_n$ is also non-empty, then there exists $g\in g$ such that
$gU_i=g_iU_i$ for $i \leq n$.
\end{itemize}
In short, $\UU$ is a regular $G$-cover if its nerve $N(\UU)$ is a regular $G$-complex.
\end{defi}

Let $\mathcal{U} =\{U_\alpha\}_{\alpha \in I} $ be an open $G$-cover of $G$-space $X$. For any subgroup
$H \subset G$ and $\alpha \in I$, let $U^H_\alpha = U_\alpha \cap X^H$. Denote by $\mathcal{U}^H$
the collection of $\{U^H_\alpha\}_{\alpha \in I}$. It is clear that $\mathcal{U}^H$ is an open cover of $X^H$.
The following definition is a natural extension of the concept of a good cover to the equivariant setting.

\begin{defi}\label{defi:G-good cover II}
A  regular open $G$-cover $\mathcal{U}$ (split into orbits $\tilde{U}= GU$)  is said to
be a \emph{good $G$-cover} if  all orbits $\tilde{U}$ of   elements  of $\mathcal{U}$  and all their
non-empty finite intersections are $G$-contractible.
\end{defi}

\begin{rema}\label{rem: good cover induces good of orbit}
It follows immediately from the definition of a good $G$-cover $\mathcal{U}$ that the set of
images $\mathcal{U}^*:=\{\pi(U)\mid U\in\UU\}$ with respect to the projection $\pi\colon X\to X/G$
forms a good cover of the orbit space $X^*=X/G$.
\end{rema}

Yang \cite{Yang} introduced another definition of good $G$-covers which
may look less natural, but is more convenient if one wants to state
and prove an equivariant version of the Nerve theorem.
In the next proposition we show that the two approaches are in fact
equivalent.

\begin{prop}\label{prop:comparison of Definitions}
A $G$-cover $\mathcal{U} =\{U_s\}$ (split into orbits $\tilde{U}_{i\in I}$) is a good $G$-cover of a $G$-manifold  $X$ if,
and only if it is
 a regular $G$-cover  and $\mathcal{U}^H$ is a good cover of $X^H$ for all subgroups $ H\subset G$.
\end{prop}
\begin{proof}
$\mathcal{U} =\{U_s\}$, split into orbits $\tilde{U}_{i\in I}$ be a good $G$-cover of $X$ in the sense of
Definition \ref{defi:G-good cover II}, and $r: \tilde{U}_{i_1} \cap \, \cdots\, \cap \tilde{U}_{i_k} \to Gx$ a
$G$-deformation retraction. Then the restriction $ r^H$  of $r$ to
$(\tilde{U}_{i_1} \cap \, \cdots\, \cap \tilde{U}_{i_k})^H$ is a deformation retract of each intersection
of  $(U_{i_1} \cap \, \cdots\, \cap U_{i_k})^H$  onto the unique  point $x\in (Gx)^H $ which is in
$U_{i_1} \cap \, \cdots\, \cap U_{i_k}$. This shows that all sets $(U_{i_1} \cap \, \cdots\, \cap U_{i_k})^H$
are contractible which shows that this $G$-cover satisfies the stated condition.

 To show the converse implication we use an adaptation of the  argument of proof of Corollary \cite[2.12]{Yang}. Let  $\mathcal{U} =\{U_{s\in S}\}$  split into orbits $\{\tilde{U}_{i\in I}\}$, $ \tilde{U}_i = {\underset{g \in G}\cup}\, g U_{s}$, be a $G$-cover of $X$ such that  $ \mathcal{U}^H :=\{U^H_{s\in S}\}$ is  a good cover  $X^H$ for all $H\le G$. Since $\mathcal{U} =\{U_{s\in S}\}$ is regular $G$-cover, $g\,U_s \cap g^\prime\,U_s \neq \emptyset$ implies $g\,U_s = g^\prime\,U_s $. A nonempty intersection

 $$ \tilde{W} = \tilde{U}_{i_1} \cap  \, \cdots \,\cap \, \tilde{U}_{i_n}$$
 is of the form
 \begin{equation}\label{different definition equation}
 {\underset{i_1, \, \dots\, \i_n }{\, \bigcup\,}}\; g_{i_1} U_{i_1}\cap \,\cdots\,\cap \, g_{i_n} U_{i_n}
 \end{equation}
 If the isotropy group $G_{g_{i_k} U_{i_k}}= H_{i_k}$ then the isotropy group of $ g_{i_1} U_{i_1}\cap \,\cdots\,\cap \, g_{i_n} U_{i_n}$ is $ H= {\underset{1}{\overset{n}\cap}}\,  H_{i_k}$, i.e.
  $ g_{i_1} U_{i_1}\cap \,\cdots\,\cap \, g_{i_n} U_{i_n}=  ( g_{i_1} U_{i_1}\cap \,\cdots\,\cap \, g_{i_n} U_{i_n})^H$.
 Since  every $U_{i_k}$, thus $ g_{i_k} U_{i_k}$, is contractible, every nonempty summand of $\tilde{W}$  in   (\ref{different definition equation}) is contractible by our assumption applied for $H$.

 Since $\tilde{W}$ is nonempty, there exists nonempty contractible space $D=  g_{i_1} U_{i_1}\cap \,\cdots\,
 \cap g_{i_n} U_{i_n}$. If $D^\prime = g^\prime_{i_1} U_{i_1}\cap \,\cdots\, \cap g^\prime_{i_n} U_{i_n} $ is 
 also nonempty, then
 the condition (RC2)  of Definition \ref{defi:regular cover} implies there exists $h\in $G  such that
  $g_k^\prime U_{i_k} = h(g_k U_{i_k})$ for $k= 1, 2,\, \dots, n$. Hence $D^\prime = hD$.

  Since $\tilde{W} $ is $G$-invariant, $ W$ has the form $G/H \times D$, where $D$ is contractible and the action of $G$ is on the first factor.

     Let $r_t: D\times [0,1] \to D$ be the deformation of $D$ to a point $[x]= r_1(D)$, where we identify   $D$ with its image $\pi(D)\subset X/G$. Since over $D\subset X/G$ there is only one orbit type, the deformation  $r_t: D\times [0,1] \to D$ lifts to a $G$-deformation  $\tilde{r}_t: \tilde{U}_{i_1} \cap \,\cdots\, \cap \tilde{U}_{i_k} \times [0,1] \to \tilde{U}_{i_1} \cap \,\cdots\, \cap \tilde{U}_{i_k}$.   This shows that this intersection  is $G$-contractible to an orbit $ Gx\simeq G/H =r_1(\tilde{U}_{i_1} \cap \,\cdots\, \cap \tilde{U}_{i_k})$.  Consequently $\mathcal{U}$  satisfies the condition of Definition \ref{defi:G-good cover II}.

The above implication can be also shown by another argument. For a $G$-cover of $X$ $\mathcal{U} =\{U_{s\in S}\}$,  split into orbits $\{\tilde{U}_{i\in I}\}$, $ \tilde{U}_i = {\underset{g \in G}\cup}\, g U_{s}$,   such that  $ \mathcal{U}^H :=\{U^H_{s\in S}\}$ is  a good cover of  $X^H$ for all $H\le G$ we take the geometric realization of its nerve  $\mathcal{N}(\mathcal{U})$.  Under this assumption the Alexandrov map $\phi: X \to  |\mathcal{N}(\mathcal{U})|$ is $G$-homotopy equivalence as follows from Theorem \ref{thm:G-nerve homotopy equivalence}. Then the $G$-cover $\tilde{W}$ of $|\mathcal{N}(\mathcal{U})|$, consisted of saturations of the stars  of vertices of  the nerve $\mathcal{N}(\mathcal{U})$, satisfies the condition of Definition \ref{defi:G-good cover II} as follows from  Corollary \cite[2.12]{Yang} and the above argument. Consequently  $\tilde{\mathcal{U}}^\prime = \phi^{-1} (\tilde{W})$ gives a  regular good $G$-cover of $X$  of the same cardinality as $\tilde{U}$ which satisfies Definition \ref{defi:G-good cover II}.
\end{proof}

Theorem 2.11 of Yang \cite{Yang} states that  every smooth $G$-manifold has a good $G$-cover.
The equivariant good covers are in fact co-final in the set of open covers of a $G$-manifold $X$.
Moreover, Yang   \cite[Theorem 2.19]{Yang} proved the following equivariant  version of the Nerve theorem.

\begin{teo}\label{thm:G-nerve homotopy equivalence}
If $\mathcal{U}$ is a locally finite, e.g. finite,  equivariant good cover of a $G$-CW complex $X$,
then the usual realization $|\mathcal{N}(\mathcal{U})|$ of the nerve $\mathcal{N}(\mathcal{U})$ is
$G$-homotopy equivalent to $X$.
\end{teo}

We are now prepared to define the main concept of this paper.

\begin{defi}\label{defi:sct & ct}
The \emph{strict $G$-covering type} of a given $G$-space $X$  is  the minimal cardinality $\sct_G(X)$
of the set of orbits of a $G$-invariant regular good cover for $X$.

Furthermore, the \emph{$G$-covering type} of a $G$-space $X$ is the minimal value of $\sct_G(Y)$ among
spaces $Y$ that are $G$-homotopy equivalent to $X$, i.e.,
 $$\ct_G(X):=\min\{\sct_G(Y)\,\mid\, Y {\overset{G}{\,\simeq\,}} X\}$$
\end{defi}

Note that $\sct_G(X)$ and $\ct_G(X)$ can be infinite (e.g., if $X$ is an infinite discrete space) or even
undefined, if the space does not admit any good covers. In what follows we will always tacitly assume that
the spaces under consideration  admit finite good covers.

It is clear from the definitions that a $G$-invariant regular open cover $\mathcal{U}$ of $X$ induces an open
good cover of the orbit space $X/G$ as the projection map $\pi: X \to X/G$ is open and $G$-contraction of
$\tilde{U}$ to an orbit $Gx$ induces a contraction of $p(\tilde{U})$ to $*=[Gx]$ in $X/G$. Thus we get
the following result

\begin{coro}\label{coro:ct(X^*) < ct_G(X)}
For a $G$-space $X$ which is a $G$-CW complex we have
$$ \sct(X/G) \leq \sct_G(X) \;\; \; {\text{and
}} \;\;\; \ct(X/G) \leq \ct_G(X)$$
\end{coro}

In the proof of Proposition \ref{prop:comparison of Definitions} we used the fact that
open stars of a regular $G$-complex $K$ are a good $G$-cover of $|K|$. In view of
Definition \ref{defi:G-good cover II} this immediately yields the following result.

\begin{prop}\label{prop:estimate vertices by covering type}
Let $K$ be a simplicial $G$-complex with a regular action of $G$. Then
   $$ \ct_G(|K|)\leq \sct_G(|K|) \leq \Delta^*(K),$$
where $\Delta^*(K)$ denotes the number of orbits of vertices of $K$.
\end{prop}

In discrete geometry, to every finite simplicial complex $K$ of dimension $d$ one can associate
the so called \emph{f-vector} $\vec{f}(K)=(f_0(K), f_1(K), \,\dots\, , f_d(K))$, where $f_i(K)$
is the number of $i$-dimensional faces of $K$. Analogously, given a $d$-dimensional
simplicial $G$-complex $K$ we define the \emph{equivariant f-vector}
\begin{equation}\label{def: G-vectors f}
 \vec{f}_G(K):= (f_{G,0}(K), f_{G,1}(K), \dots  , f_{G,d}(K)) \;\;
\end{equation}
where $f_{G,i}(K)$ is the number of orbits of $i$-dimensional  simplices of $K$.

Note that the $f$-vector and the equivariant $f$-vector of a simplicial $G$-complex are related by the
formula
$$f_i(K)= {\underset{\sigma_i}\sum} \, |G/G_{\sigma_i}| =
{\underset{1}{\overset{f_{G,i}}{\,\sum}}}\,|G/G_{\sigma}|,$$
where the sum is taken over representatives of all orbits of $i$-simplices $\sigma$ of $K$ or equivalently
of all $i$-simplices of the induced triangulation of $K/G$.

The main aim of this paper is to give some lower estimates of $f_{G,0}(K)$ and $f_{0}(K)$.

\subsection{$G$-covering type of one-dimensional complexes}
\label{$G$-covering type of one-dimensional complexes}
Our first explicit example is the computation of the $G$-covering type of a finite $G$-graph, i.e.,
a space $X=|K|$, where $K$ is a 1-dimensional simplicial complex with a regular simplicial action
of a group $G$. In other words, $G$ permutes vertices and edges of $K$ and for every edge $e=[v_1,v_2]$,
if $g[v_1, v_2] \subset [v_1,v_2]$, then $g={\rm id}_{[v_1,v_2]}$ .

Before proceeding, let us recall the corresponding result for the non-equivariant case
(cf. Karoubi-Weibel \cite[Proposition 4.1]{K-W})\newline
{\it If $X_h$  is a bouquet of $h$ circles then
	\begin{equation}\label{first bouquet}\ct(X_h) =\Bigg\lceil\,\frac{3+\sqrt{1+8h}}{2}\,\Bigg\rceil
	\end{equation}
	where $ \lceil \alpha \rceil$ denotes the ceiling of  a real number $\alpha$. In other words, $\ct(X_h)$ is the unique integer $ n$ such that
	\begin{equation}\label{second bouquet} \binom{n-2}{2} < h \leq \binom{n-1}{2}
\end{equation}}

Let $X$ be a finite $G$-graph and let $\UU$ be a regular $G$-cover of $X$. As before, denote by $\tilde\UU^*$ the
induced cover of $X/G$. Suppose first that $X=X_{(H)}$, i.e., we have one orbit type $(H)$. In this case we have
$\ct_G(X)= \ct(X/G)$, because every deformation of $\tilde{U}_{i_1}^*\cap\cdots\cap \tilde{U}^*_{i_k}$ of
projections $\tilde{U}^*_{i_j}= \pi(\tilde{U}_{i_j})$ to a point $[x]=\pi(x)\in X^*=X/G$ in
$\tilde{U}_{i_1}^*\cap\cdots\cap\tilde{U}^*_{i_k}$ can be lifted to  a $G$-deformation of
$\tilde{U}_{i_1}\cap\cdots\cap\tilde{U}_{i_k}$ to the orbit $Gx = G/H$.
Observe that  topological loops in $X^*$ come from two different types of situations:
\begin{itemize}
\item[i)] {if $|G/H|=  m$ and $G$ acts as the  cyclic group of order $m$ on a cycle of the form
$$[v_1, v_2],[v_2,v_3],\ldots,[v_m,v_1],$$
with the action $v_i\mapsto v_{i+1} \mod m$.}
\item[ii)] {if the loop is given as the graph cycle $[w_1, w_2], [w_2,w_3],\dots ,
[w_n,w_1]$ and $gw_i\notin \{w_1, \dots , w_n\}$ for any $g\in G$. Consequently, for every $g\in H$ such
that $[g]\neq [H] \in G/H$ the image of the cycle $[w_1, w_2], [w_2,w_3], \dots , [w_n,w_1]$ is another
disjoint cycle  $[w^\prime_1, w^\prime_2], [w^\prime_2,w^\prime_3] \dots , [w^\prime_n, w^\prime_1]$.  The
group $G$ is acting trivially on elements of each such cycle and is permuting these cycles, i.e., each
cycle can be identified with an element of $G/H$ and $G$ is permuting them transitively as it permutes
elements of $G/H$. }
\end{itemize}

Consequently, every 1-dimensional finite regular $G$-graph is, up to $G$-homotopy, a union of topological
loops of type i) or ii). Assume there are $k$ cycles  of type i), and $l \times |G/H|$ cycles (with the
corresponding action) of type ii). This shows that $\ct_G(X)=\ct(X^*)= k +l$, i.e., the number of loops in the
bouquet $X_h$ of $h=k+l$
circles being homotopy equivalent to $X^*=X/G$. Applying (\ref{first bouquet}),   as a result we get a formula
for $G$-covering type of a finite graph $X$ such that $X=X_{(H)}$
\begin{equation}\label{third bouquet}
\ct_G(X)= \bigg\lceil\,\frac{3+\sqrt{1+8(k+l)}}{2}\,\bigg\rceil,
\end{equation}
with $X$ having  $k$ cycles of type i) and $l$ orbits of cycles of type ii).

\zz{Now suppose that $X= (X\setminus \{*\}) \cup *$ where $\{*\}= \{*\}^G$ is a fixed point of the action and
$  (X\setminus \{*\})= (X\setminus \{*\})_{(H)}$ is of one orbit type set with $H\neq G$.
If $X$ is as above and $ e_1, \, \dots\, e_p$ be all edges outgoing (or equivalently ingoing) from $\{*\}$
then by $X^\prime_{(H)}$ we denote the compact closed set (graph) $X\setminus N_\epsilon(\{*\})$. Let next
$0\leq h_{(H)}(X)$, shortly $h_{(H)}$, be the number of loops of $ X^\prime_{(H)}/G$, i.e.the number of
generators of $\pi_1(X^\prime_{(H)}/G)$. In this case we have the same formula (\ref{third bouquet}) for
$\ct_G(X)$ as in the previous case.}

We handle the general case by  using induction with respect to  the partially ordered
set of all orbit types $\mathcal{S}_G$ with the order $ (H) \succ (K) \, \Longleftrightarrow \, K \subset
gHg^{-1}$ for some $g\in G$. For a given $G$-space $X$, by $\mathcal{S}_G(X)$ we denote the subset of
$\mathcal{S}_G$ consisting of  $(H)$ for which $X_{(H)}\neq\emptyset$. Obviously,  we can restrict our consideration
to $(H)\in \mathcal{S}_G(X)$.

Let $N_\epsilon(A)$ denote open and invariant  neighbourhood   of invariant set $A$. If  $X$ is a $G$-graph
and $A {\overset{G}\subset} X$ is an invariant subcomplex
then as $N_\epsilon(A)$ we can take the union of open stars of all vertices of $A$, i.e. the union of all
edges of $ A$ and half-open edges of length $\epsilon$ of $X$ which have one vertex in $A$. In the case of
$G$-graphs we can take always $\epsilon=1$, i.e. the open star $N_1(A)={\rm st}(A)$.

In the inductive step, suppose that $X=X^{(H)}= X_{(H)}\sqcup X^{(K)\nsucc (H)}$.  In general there exists
$\epsilon >0$ such that $X^{(K)\nsucc (H)}$ is a $G$-deformation retract of  the open and invariant
neighbourhood $N_\epsilon(X^{(K)\nsucc (H)})$, but here it holds for $N_1(X^{(K)\nsucc (H)})$.

Let $X^\prime_{(H)}$ be the compact closed set (graph) $X\setminus N_\epsilon(X^{(K)\nsucc (H)})$, and let
$0\leq h_{(H)}(X)$, or shortly $h_{(H)}$, be the number of loops of $ X^\prime_{(H)}/G$, i.e. the number of
generators of $\pi_1(X^\prime_{(H)}/G)$.

We can now state the main result of this subsection.
\begin{theorem}\label{teo: ct_G of graph}
Let $X$ be a finite connected  graph with a regular simplicial action of a group $G$, and let
$\mathcal{S}_G(X):=\{(H)\in \mathcal{S}_G\mid X_{(H)}\neq \emptyset \}$. Then
$$ \ct_G(X)\, =\, \sum_{(H)\in   \mathcal{S}_G(X)} \, \Bigg\lceil\,\frac{3+\sqrt{1+8h_{(H)}}}{2}\,\Bigg\rceil
\,.	$$
\end{theorem}
\begin{proof}
Let $(H)$ be a minimal orbit type in $\mathcal{S}_G(X)$. For $X=X^{(H)}= X_{(H)}$ we have already shown the
statement, cf. (\ref{third bouquet}). Suppose that $X=X^{(H^\prime)}= X_{(H^\prime)} \sqcup X_{(H)}$, with
$X_{(H)}=X^{(H)}$, where $(H^\prime)\succ (H)$ and that there does not exists $(K)\neq (H^\prime), (H)$ such that $(H^\prime)\succ (K)\succ (H)$.
	
Let $\mathcal{U}=\{\tilde{U}_i\}$,  $1\leq i \leq \ct_G(X)$,  be a regular good $G$-cover of $X$ consisting of
$G$-contractible saturations of open stars of vertices. Decompose $\mathcal{U} = \mathcal{U}_1 \sqcup
\mathcal{U}_2$ into two subfamilies of cardinality $c_1,\, c_2$ respectively, where $\mathcal{U}_1$ consists
of $\tilde{U}_i$ for which $\tilde{U}_i \cap X^{(H)} \neq \emptyset$, and  $\mathcal{U}_2$ consists of
$\tilde{U}_i$ for which $\tilde{U}_i \cap X^{(H)} =\emptyset$. Note that if
$\tilde{U}_i \cap X^{(H)} \neq \emptyset$
then $\tilde{U}_i $ is $G$-contractible to an orbit $Gx$, where $x$ is a vertex in $X^{(H)}$, because
the contraction
is a $G$-map. Moreover, ${\underset{j=1}{\overset{c_2} {\,\cup\,}}} \tilde{U}_j = X_{(H^\prime)}$, because
$X_{(H^\prime)} = X^{(H^\prime)}\setminus X^{(H)}$ and $X^{(H)}$ consists of all edges, such that
both vertices are in $X^{(H)}$.
	
By Proposition \ref{prop:comparison of Definitions}, sets $\{\tilde{U}_i\}_{i=1}^{c_1}$ form a regular good
$G$-cover of $ X^{(H)}$ thus $c_1 \geq \ct_G(X^{(H)})$.
Suppose that $c_1 < \ct_G(X_{(H)}) = \ct_G(X^{(H)})$. Then there a exists a good $G$-cover
$\{\tilde{U}_i^\prime\}_{i=1}^{c^\prime_1}$ of  $ X^{(H)})$  with $ c^\prime_1 < c_1$. Let
$r: N_1(X^{(H)}) \to X^{(H)}$ be a $G$-deformation retraction. Then
$\{\tilde{U}_i^{\prime\prime}\}_{i=1}^{c^\prime_1}= r^{-1}(\{\tilde{U}_i^\prime\}_{i=1}^{c^\prime_1}$ is a
good $G$-cover of $N_1(X^{(H)})$, because we can compose the $G$-deformation retraction
$r: N_\epsilon(X^{(H)})\to X^{(H)}$ with corresponding $G$-contractions of $\{\tilde{U}_i^\prime\}$.
Since ${\underset{i=1}{\overset{c_1^\prime} {\,\cup\,}}} \tilde{U}_i^{\prime\prime}\cup {\underset{j=1}{
\overset{c_2} {\,\cup\,}}} \tilde{U}_j = N_\epsilon(X^{(H)}) \cup X_{(H_\prime)} = X$, it gives a good $G$-
cover of $X$ of cardinality $c_1^\prime + c_2 < c_1+ c_2= \ct_G(X)$ which leads to a contradiction.
	
Now we have to show that $c_2= \ct_G(X_{(H^\prime)})$. Since  $\{\tilde{U}_j\}_{j=1}^{c_2}$ form  a good $G$-
cover of $X_{(H^\prime)}$ we have $c_2 \geq \ct_G(X_{(H^\prime)})$.
If $c_2 >  \ct_G(X_{(H^\prime)})$ then we can take a good $G$-cover
$\{\tilde{U}_j^\prime\}_{j=1}^{c^\prime_2}$ with $c_2> c_2^\prime = \ct_G(X_{(H^\prime)})$.
The union of elements  $\{\tilde{U}_i\}_{i=1}^{c_1}$ and $\{\tilde{U}_j^\prime\}_{j=1}^{c^\prime_2}$ gives a
good $G$-cover of $X^{(H)}\sqcup X_{(H^\prime)} = X$ of cardinality
$c_1+c_2^\prime < \ct_G(X)$ which leads to a contradiction. Consequently $c_2=  \ct_G(X_{(H^\prime)})$. This
shows the statement for $X$ having two different orbit types which are comparable.
	
If  in $X$ there are more than one   minimal orbit types $(H_1), \dots , (H_r)$  and another orbit type
$(H^\prime)\succ (H_i)$ for all $1\leq i \leq r$ then the same argument shows the statement for
$X=X^{(H^\prime)}$. Notify that if $X$ consists of only  different minimal orbit types (not comparable)  then
it is not connected.
	
By induction over orbit types assume that the statement holds for $G$-graphs $Y$ having at most $n$ different
orbit types $(H_1), \, \dots \, (H_n)$ and $X=Y \sqcup X_{(H_{n+1})}$. By the above remark and connectedness
of $X$ we can assume that $(H_{n+1})$ is not minimal. By repeating  the argument of the first inductive step
we get $\ct_G(X) = \ct_G(Y) + \ct_G(X_{(H_{n+1})})$.
Using the induction assumption we get
$$ \ct_G(X) =  \ct_G(Y) + \ct_G(X_{(H_{n+1})} ) = \sum_{i=1}^n \, \Bigg\lceil\,\frac{3+\sqrt{1+8h_{(H_i)}}}{2}
\,\Bigg\rceil + \Bigg\lceil\,\frac{3+\sqrt{1+8h_{(H_{n+1})}}}{2}\,\Bigg\rceil \,$$
as claimed.	
\end{proof}

\section{Equivariant covering type and  minimal triangulations of surfaces}
\label{Equivariant covering type and  minimal triangulations of surfaces}

In this section we determine the equivariant covering type of a closed orientable surface $\Sigma_\g$
of genus $\g$ with respect to an orientation-preserving actions of a finite group $G$.
As a consequence, we are able to estimate the minimal number of vertices and $G$-orbits of vertices
in a triangulation of $\Sigma_\g$ by a regular simplicial $G$-complex.

\subsection{Orientation preserving actions on orientable surfaces}
Let $\Sigma_\g$ be oriented surface of genus $\g \geq 0$. Suppose that $G$ acts {\bf effectively} on
$\Sigma_\g$ {\bf preserving orientation}, i.e., it is a (finite) subgroup of ${\rm Homeo}_+(\Sigma_\g)$.
It is known (Hurwitz for $\Sigma_\g$ with $\g\geq 2$,   Brouwer, Kerekjarto and Eilenberg for $\Sigma_\g=S^2$
and folklore for $\Sigma_\g=\bt^2$) that there exists a holomorphic structure $\mathcal{H}$ on $\Sigma_\g$
in which ${\rm Homeo}_+(\Sigma_\g)$ can be viewed as a subgroup of biholomorphic isomorphisms
${\rm Hol}(\Sigma_\g, \mathcal{H})$ of $(\Sigma, \mathcal{H})$.
More precisely we have the following theorem
\begin{theorem}[Geometrization of action]\label{thm:Edmonds}
Given a finite group $G$ of orienta\-tion-preserving
homeomorphisms of a compact orientable surface $X$ of genus ${\rm g}$, there is a complex structure on $X$
with respect  to which $G$ is a  group of its conformal maps. Furthermore, the orbit
space $X'=X/G$ is a compact surface of genus $\g' \leq  \g$ and the relation between $\g$ and $\g'$ is given
by the Riemann-Hurwitz formula (see formulas (\ref{Riemann-Hurwitz}) and (\ref{converse Riemann-Hurwitz})
below).
\end{theorem}

Moreover, Hurwitz' theorem says that the for $\g \geq 2$ the order of ${\rm Hol}(\Sigma_\g, \mathcal{H})$ is
$\leq 84(\g -1)$.

Let $\Sigma_\g$ be a compact Riemann
surface of genus ${\rm g}\geq 0$ and let $G$ be a group of holomorphic
automorphisms of $\Sigma_\g $. Let $ \Sigma_{\g\prime} =  \Sigma_\g/G$ be the quotient surface
of genus $\g^\prime$ with the projection $\pi:X \to X^\prime $ and
let $\{x^\prime_1,\, \ldots , x^\prime_r\}$ be the set of all points
over which $\pi$ is branched, i.e. the image of the singular orbits.
For any $x_{i,j}\in \pi^{-1}(x^\prime_j)$ being a point in the orbit
over $x^\prime_j$, the isotropy group $G_{x_{i,j}}$ is cyclic of
order $m_j$ (observe that the isotropy groups of points in one orbit
are conjugate). We denote by $m$ the order  of $G$. Now, it follows
that the orbit of $x_{i,j}$ has order $n_j={m}/{m_j}$. If
$x^\prime\in \Sigma_{g^\prime}\setminus \{x^\prime_1, \ldots ,
x^\prime_r\}$, then $\pi^{-1}(x^\prime)$ is an orbit isomorphic to
$G$ and consequently, for every $x\in \pi^{-1}(x^\prime)$,
$m_x=1$ and $n_x={m}/{m_x} = m$.

Denote by $\mathcal{S}$  the set of images of singular orbits $\{x_1^\prime, \dots, x_r^\prime\} $ in
$\Sigma^\prime$.
We have the classical Riemann-Hurwitz formula
\begin{equation}\label{Riemann-Hurwitz}
	\g= 1 + m(\g^\prime-1) +
	\frac{1}{2}\, m\, {\underset{j=1}{\overset{r}\sum}}\, (1
	-\frac{1}{m_j}),
\end{equation}
which let us also express $\g^\prime$ as a function of $\g$.
\begin{equation}\label{converse Riemann-Hurwitz}
	\g^{\prime} = 1 + \frac{1}{m}(\g-1) -
	\frac{1}{2}\,\big( {\underset{j=1}{\overset{r}\sum}}\, (1
	-\frac{1}{m_j})\big)\,.
\end{equation}
Riemann-Hurwitz formula gives relations between $\g,\, \g^\prime\,,r, \, m\,,\{m_1\,,m_2\,,\dots\,,m_r\}$,
where
$\g\geq 0$, $m=|G|$,  $2\g +2 \geq r\geq 0$, $m_j\mid m$, i.e., some necessary condition on the action of $G$
on $\Sigma_\g$. The system $ (\g^\prime\,, r,\,  \{m_1\,,m_2\,, \, \dots\,, m_r\}) $ is called the
\emph{generating vector},
provided a condition is satisfied (cf. \cite[Definition 2.2]{Broughton}).
The following classical classical result provides a converse to the Riemann-Hurwitz formula
(see \cite[Proposition 2.1]{Broughton} for more information).

\begin{prop}[Riemann’s Existence Theorem]\label{prop:Riemann’s Existence Theorem}
The group $G$ acts on the surface $\Sigma_\g$,
of genus $\g$, with branching data $(\g^\prime, r,  m_1, \dots , m_r,)$ if and only if the Riemann-Hurwitz
equation (\ref{Riemann-Hurwitz})
is satisfied, and $G$ has a generating
$(\g^\prime\,, r\,, m_1,\, \dots\,, m_r)$-vector.
\end{prop}

\subsection{Minimal triangulations of orientation preserving actions on surfaces}

We take as a starting point the famous formula of Jungermann and Ringel that gives the minimal number of
vertices in a triangulation of a closed surface in terms of its Euler characteristic.

\begin{theorem}[Jungerman and Ringel]\label{thm:J-R}
Let $S$ be a closed surface different from the orientable surface of genus $2$ (M2), the Klein bottle (N2)
and the non-orientable surface of genus $3$ (N3). There exists  a triangulation of $S$ with
$n$ vertices if, and only if
$$
n\ge\left\lceil \frac{7 + \sqrt{49- 24\,\chi(S)}}{2} \right\rceil\,.
$$	
If $S$ is any of the exceptional cases (M2), (N2) or (N3), then the right-hand side of the formula
should be replaced by
$ \left\lceil\frac{7 + \sqrt{49- 24\chi(S)}}{2} \right\rceil +1$.
\end{theorem}

For orientable surfaces we have $\chi(\Sigma_g)=2-2g$, so we let
$n_g:=\left\lceil \frac{7 + \sqrt{1+48 g}}{2} \right\rceil$ for $g\ne 2$ and $n_2:=10$.
Our goal is to obtain a similar formula that estimates the minimal number of  orbits of vertices of a
triangulation of $\Sigma_\g$ which admits a simplicial, regular and orientation-preserving action of
a group $G$. The basic idea
is quite simple. Let $\Sigma_{\g^\prime}$ be the orbit space $\Sigma_\g/G$ and let $K^\prime$ be its minimal
triangulation. If the action of $G$ is free, then the projection $\pi: \Sigma_\g \to
\Sigma_{\g^\prime}= \Sigma_\g/G$ is a regular covering, so it is enough to lift a minimal triangulation
$K^\prime$ of $ \Sigma_{\g^\prime}$. If the action $G$ is not free, then in order to lift $K^\prime$
to a regular invariant triangulation $K$ of $\Sigma_g$ we have to check a compatibility condition
given by the Riemann-Hurwitz formula.

Let $H$ be a non-trivial subgroup of $G$. Then the action of $H$ has a non-trivial fixed point set
$ (\Sigma_\g)^H$ if, and only if $H$ is a cyclic. Moreover, the set${\Sigma_\g^{(H)}}$ is a union of
finitely many isolated orbits, therefore ${\Sigma^{(H)}}^*$ consists of a
finite number of points. Since the action is regular, ${\Sigma^{(H)}}^*$ is a subcomplex of any lift
$K$ of $K^\prime$.  Consequently, the first condition of compatibility is satisfied if every
$x^\prime \in \mathcal{S}$ is a vertex of the triangulation $K^\prime$, i.e.
$\mathcal{S} \subset (K^\prime)^{(0)}$.

If $v_1^\prime, \,v_2^\prime$ are two vertices of $K^\prime$ with
$<v_1^\prime, \,v_2^\prime>\in (K^\prime){(1)}$ an edge joining them, and if at least one of the vertices
is not in $\mathcal{S}$, then there is a lift of $<v_1^\prime, \,v_2^\prime>$ to an invariant $1$-dimensional
subcomplex of the saturation $G <v_1, \,v_2>$  of any lift $<v_1, \,v_2>$ of $<v_1^\prime, \,v_2^\prime>$. If
$v_1^\prime, \,v_2^\prime \in \mathcal{S} $, then a lift of $K^\prime$ is not a regular $G$-complex in general.

\begin{lema}\label{sufficient for lift}
Let $K^\prime$ be a triangulation of $\Sigma_{\g^\prime}$ such that $\mathcal{S} \subset (K^\prime)^{(0)}$,
and assume that the vertices $v_i, v_j \in \mathcal{S}$ do not span a simplex in $K^\prime$.
Then there exists a lift of $K^\prime$ to a triangulation  $K=\pi^{-1}(K^\prime)$ of $\Sigma_\g$ in which
the action of $G$ is simplicial and regular.
\end{lema}
\begin{proof}
The existence of a lift is clear for the vertices. The lift of all them is a union of orbits.  Consider a
$G$-invariant Riemannian  metric  on $\Sigma_\g$.  It induces a Riemaniann metric on $\Sigma_\g^\prime$ thus
it can be considered as a lift of the latter. If $e^\prime$ is an edge joining the vertices
$v_1^\prime, \, v_2^\prime$, then we can lift $v_1^\prime$ to $ v_1\in \Sigma_\g$ and take as a lift $v_2$
of $v_2^\prime$ the closest point of the orbit $\pi^{-1}(v^\prime_2)$. Define a lift of
$<tv^\prime_1,(1-t) v^\prime_2>$ as $<tv_1,(1-t) v_2>$. This gives a lift of $e^\prime $ to an edge $e$
joining $v_1$ and $v_2$.  Its orbit $G \,e = G\, <v_1, v_2>$ is $G$-homeomorphic to
$ G \times_G <v_1, v_2>$. The latter is a disjoint union of $|Gw$ edges if $ G_{v_1}= G_{v_2}= {\bf e}$.
If $G_{v_i} = \bz_m$ for  one of $i=1, 2$, then it is a union of $\vert G\vert $ intervals grouped in
$\vert G/G_{v_i} \vert $ disjoint sub-collections each consisting of $\vert G_{v_i} \vert $ edges intersecting
at a single point of the orbit of $v_i$. This defines a  triangulation $\pi^{-1}(<v^\prime_1, v^\prime_2>)$
in which the action is simplicial and regular. Applying this procedure consecutively to all edges of
$K^\prime$ we get a graph $K^{(1)}$ on which the action of $G$ is simplicial and regular.

Now we have to lift triangles of $K^\prime$. If $\Delta^\prime = <v^\prime_1, v^\prime_2, v^\prime_3>$ is a
triangle of $K^\prime$, then at most one of the vertices has a non-trivial  isotropy group by our assumption.
If all $ v^\prime_i \in \Sigma^\prime \setminus \mathcal{S}$ then  $\pi$ over $\Delta $ is a regular covering,
thus there exists a unique lift $\Delta=<v_1, v_1,v_3>$ of $\Delta^\prime$ such that
$\pi^{-1}(\Delta^\prime)$ $G$-homeomorphic to $G \times \Delta$ as a $G$-set. If for
$\Delta^\prime = <v^\prime_1, v^\prime_2, v^\prime_3>$ the isotropy group of one vertex is nontrivial, say
$G_{v_1} = \bz_m$, then we have the lifts of the edges $<v^\prime_1, v^\prime_2>$ and
$<v^\prime_1, v^\prime_3>$  to $<v_1, v_2>$, and $v_1,v_3.$ respectively,  by the procedure described above.
The edge $<v^\prime_2, v^\prime_3>$ has a unique lift $<v_2,v_3>$ because $\pi$ is a regular covering over
$<v^\prime_2, v^\prime_3>$. This gives  a lift  of $\Delta^\prime$ to $\Delta = <v_1,v_2,v_3>$. In this case
the saturation $G \Delta$ is $G$-homeomorphic to $G\times\Delta $, a union of $\vert G\vert $
triangles that are divided in $\vert G/G_{v_1} \vert $ disjoint sub-collections each consisting of
$\vert G_{v_i} \vert $ triangles intersecting at a single point of the orbit of $v_1$. Moreover, the group
 $G_{v_1}$  acts trivially on all triangles that meet at this point.

Consequently we get a triangulation of $\Sigma_\g$ in which the action of $G$ is simplicial and regular
(even strictly regular), and the induced quotient triangulation $K/G$ is equal to $K^\prime$.
\end{proof}

Next we show the following lemma about a homogeneity of  triangulations of a  manifold with respect
to the position of vertices.

\begin{lema}\label{homogenity of triangulation}
Let $M$ be a smooth manifold of dimension $d \geq 2$  and $K$ a triangulation of $M$ with $n$ vertices
$v_1, \, v_2, \, \dots,\, v_n$. Then for every set $x_1, \, x_2,  \, \dots\,,x_n \in M $
of pairwise pairwise different points in M there exists a triangulation  $K^\prime$ of $M$ with vertices
$\{x_1, \, x_2,  \, \dots\,,x_n \}$ such that $K^\prime$ is isomorphic to $K$.
\end{lema}
\begin{proof} The statement follows from  the fact that  a smooth manifold of $\dim \geq 2$
is $n$-tuples homogeneous for every $n\geq 1$ (cf. \cite{Mich-Viz}). This means that  for every two pairs of
$n$-tuples $\{y_1, y_2, \, \dots, y_n\}$ and $\{x_1, x_2, \, \dots\,, x_n\}$ points of  $M$
there exists a diffeomorphism $\phi: M \to M$ such that
$\phi(y_j)= x_j$. It is enough to take as $y_i\in M$ the images of $v_i \in K$ and the composition of the
triangulating  map with $\phi$.
\end{proof}

Before the formulation of the main theorem of this subsection we need new notation.
Let $\Sigma_{\g^\prime}=\Sigma_\g/G$ be the orbit surface of an orientation-preserving action of a group $G$ with the
branching data $(\g^\prime, r,  m_1, \dots , m_r,)$.
Let ${\bf n} = \max\{n_{\g^\prime}, r\}$,  and let $K^\prime({\bf n}) $ be a  triangulation of
$\Sigma_{\g^\prime}$  with ${\bf n}$ vertices given by Theorem \ref{thm:J-R}.

\begin{theorem}\label{thm:surface minimal invariant triangulation}
Let $\Sigma_\g$  be an oriented surface of genus $\g$ with an orientation-preserving action by a finite
group $G$. Denote by $\Sigma_{\g^\prime} = \Sigma_\g/G $ the quotient surface and let
$r= \vert \mathcal{S}\vert $ be the number of singular fibers of the projection $\pi:\Sigma_\g \to
\Sigma_{\g^\prime}$. Moreover, let be $n_{\g^\prime}$ the number  defined in Theorem \ref{thm:J-R}, and
let ${\bf n} = \max(n_{\g^\prime}, r)$.
	
Then we have the following estimate for the number  of orbits of minimal regular $G$-triangulation $K$
of $\Sigma_\g$:
$$  {\bf n}\leq  f_{G,0}(K) \leq {\bf n} +  \binom{r}{2} + \binom{r}{3}
 \,, $$
with a convention that $\binom{r}{i}=0$ if $0\leq r < i$.

\zz{And  the number of its vertices   is estimated by
$$  \sum_{j=1}^r \,\frac{m}{m_j} \,\, + (n_{\g^\prime}-r) \, m \;\leq f_0(K)\leq  \; \sum_{j=1}^r \,\frac{m}{m_j} \,\, + (n_{\g^\prime}-r +n_{\g^\prime}^1(r)) \, m\,, \;\;{\text{if}} \;\; r \leq  n_{\g^\prime}\,,\;\;{\text{or}} $$
$$  \sum_{j=1}^{r} \,\frac{m}{m_j} \;\leq f_0(K)\leq  \; \sum_{j=1}^{r} \,\frac{m}{m_j} + n_{\g^\prime}^1(r) \, m \;\; {\text{if}}\;\; r > n_{\g^\prime}\,. $$}
\end{theorem}

\begin{proof}[Proof of Theorem \ref{thm:surface minimal invariant triangulation}]
Let $K^\prime({\bf n}) $ be a  triangulation of $\Sigma_{\g^\prime}$  with $n$ vertices given by
Theorem \ref{thm:J-R}. Using Lemma \ref{homogenity of triangulation} we can assume that  $r$ vertices of
$K^\prime({\bf n})$  are $r$-points of $\Sigma_{\g^\prime}$ over which $\pi$ is branched. We have to  expand
$K^\prime({\bf n}) $ to a triangulation $K^{\prime\prime}({\bf n}) \supset K^\prime({\bf n}) $ satisfying the
assumption of Lemma  \ref{sufficient for lift} by adding extra vertices, edges and triangles if necessary.

Let $K^\prime_{min}({\bf n}) \subset K^\prime({\bf n})$ be the minimal  subcomplex of $K^\prime({\bf n})$
containing all   $r$-points of $\Sigma_{\g^\prime}$ over which $\pi$ is branched, i.e. the subcomplex spanned
by all such vertices.

We first assume that $r> 3$. If $r<n_{\g^\prime}$ then $K^\prime_{min}({\bf n}) $ is a proper subcomplex of
$ K^\prime({\bf n})$, if $r\geq n_{\g^\prime}$ then $K^\prime_{min}({\bf n}) = K^\prime({\bf n})$.
Consider the barycentric subdivision $ L:=(K^\prime_{min}({\bf n}))^\prime $  of $K^\prime_{min}({\bf n})$
and observe that all new vertices of $L$ , i.e. the midpoints of edges and barycentres of triangles are in
$ \Sigma_{\g^\prime} \setminus \mathcal{S}$. If $r\geq n_{\g^\prime}$ then we take
$ K^{\prime\prime}({\bf n})= L$.

If $r<n_{\g^\prime}$ then we have to expand the simplicial complex structure of $L$  to the required complex
structure $ K^{\prime\prime}({\bf n}) $ of $\Sigma_{\g^\prime}$. Let
$v_0 \in K^\prime({\bf n}) \setminus K^\prime_{min}({\bf n}) $ and $\Delta = <v_0, v_1, v_2> $ be a triangle
containing it. If $  \Delta \cap K^{\prime\prime}({\bf n}) = \emptyset $ or
$\Delta \cap K^{\prime\prime}({\bf n}) = \{v_i\}$,  $ i=1, 2 $ then we leave $\Delta$ not modified.
On the other hand, if $ \Delta \cap K^{\prime\prime}({\bf n}) = <v_1, v_2>$, then the barycentre
$\tilde{v}=(\frac{1}{2}v_1, \frac{1}{2} v_2)$ is a vertex of $L$. In this case we add to $K^\prime({\bf n})$
the edge $<v_0, \tilde{v}>$ and consequently two new triangles for which it is the edge.
In this way we extend the simplicial structure of $L$ to a new triangulation $K^{\prime\prime}({\bf n})$ of
$\Sigma_{\g^\prime}$ satisfying the assumption of Lemma  \ref{sufficient for lift}.
Observe, that  for this triangulation we add new vertices only for the barycentric subdivision $L$ of
$K^\prime_{min}({\bf n})$. But the number of vertices of $L$ is equal to the number of vertices of
$K^\prime_{min}({\bf n})$, number of its edges, and number of its faces, which is  smaller or equal to  $r$,
$\binom{r}{2}$, and  $\binom{r}{3}$ respectively. This gives the upper  estimate of statement if
$r\geq n_{\g^\prime}$. If $r < n_{\g^\prime}$ then we have to add remaining $n_{\g^\prime} - r$ vertices of
$K^\prime({\bf n})$ which gives the upper  estimate by
$ (n_{\g^\prime} - r) + r +  \binom{r}{2} + \binom{r}{3} = {\bf n} +  \binom{r}{2} + \binom{r}{3}$.
It shows the estimate from above.

If $r=0$ then the projection $\pi: \Sigma_\g \to \Sigma_{\g^\prime}$ is a regular cover, and the inequality
becomes the equality. Similarly, if $r=1$ then $K^\prime_{min}({\bf n})$ is a point and the condition of
Lemma \ref{sufficient for lift} is satisfied for $K^\prime({\bf n})$, so that once more for  the upper
 estimate holds.

If $r=2$ then $K^\prime_{min}({\bf n})$ is either  a union of two points and the condition of
Lemma \ref{sufficient for lift} is satisfied, or an edge then we have to add only its center to get $L$.
Consequently, the inequality of statement holds also in this case. Finally, if $r=3$ then
$K^\prime_{min}({\bf n})$ is either the union of three points, the union of a point and an interval, or
the connected union of two edges, or a triangle. Consequently to get $L$ we have to add  to
$K^\prime_{min}({\bf n})$ either one point or two points, or four points   $= \binom{3}{1} + \binom{3}{2}$
points which once more confirms the statement.

The estimate from below is obvious, since a simplicial regular action on $\Sigma_\g$ induces a triangulation of
$\Sigma_{\g^\prime}$.
\end{proof}

\begin{proposition}\label{estimates of ct_g for surfaces}
With the same notation and assumptions as in Theorem \ref{thm:surface minimal invariant triangulation} a
similar (but weaker)  estimates hold for $\ct_G(\Sigma_\g)$, assuming that $g^\prime\ne 2$:
$$ n_{\g^\prime} \leq \ct_G(\Sigma_\g) \leq {\bf n} +  \binom{r}{2} + \binom{r}{3}
 \, . $$
Moreover, if $g^\prime=2$, then $\ct_G(\Sigma_g)\ge 9$.
\end{proposition}
\begin{proof} The lower estimates follow from the inequality
$ \ct_G(\Sigma_\g)\geq \ct(\Sigma_{\g^\prime}) = n_{\g^\prime}$ if $\g^\prime \neq 2$.
If $\g^\prime =2$, then by \cite[Proposition 3.4]{Bor-Min}we have
$\ct(\Sigma^\prime_{\g^\prime}) = 9$.

The upper estimate follows from $\ct_G(X) \leq f_{G,0}(K)$ for any regular triangulation $K$
of $X$.
\end{proof}

We complete this section with a comparison of the values of the $G$-covering type and  Lusternik-Schnirelman $G$-category for
an orientation-preserving action of $G$ on a surface $\Sigma_\g$. We have

\begin{theorem}\label{G-category of surfaces}(\cite[Theorem 4.11]{GroJezMar})
Let a finite group $ G $ act  on $X=\Sigma_\g$ through orientation-preserving homeomorphisms. Then the set of all points  is either the union
of $r \geq  1 $ orbits or empty and we have: either $\cat_G(X) = cat(X/G) = 3$ and $\g^\prime >0$ if the
action is free, or
$$\cat_G(X) =
\begin{cases}\; r\;\; {\text{ if}}\;\;  r \geq 3\,,\;\; {\text{or if}} \;\;r = 2\;\;
{\text{and}} \;\; \g^\prime = 0\,,\cr
\;3 \;\;{\text{if}}\;\; r \leq 2 \;\; {\text{and}}\;\; \g^\prime \geq 1\,,\;\;  \end{cases}\;\;$$
if the action is not free.
\end{theorem}

\begin{rema}
Note that if the action of $G$ on $\Sigma_\g$ is free, then $r=0$, and the estimate of Theorem \ref{thm:surface minimal invariant triangulation}
reduces to the equality $f_{G,0}(\Sigma_\g) = n_{\g^\prime}$. Since $n_{\g^\prime}=\ct(\Sigma_{\g^\prime})$ (except for the case $\g^\prime = 2$,
cf. \cite{Bor-Min}), we conclude $\ct_G(\Sigma_\g)=\ct(\Sigma_\g/G)= n_{\g^\prime}$.

Moreover, it is known that there is a free action of the cyclic group $G=\bz_m$ on $\Sigma_\g$ if, and only if $m$ divides $2-2\g$, and
then $2-2\g^\prime = \frac{2-2\g}{m}$. This shows that $\g^\prime$ and $n_{\g^\prime}$ can be arbitrary large if $\g$ is large.

On the other hand,  $\cat_G(\Sigma_\g)= \cat(\Sigma_{\g^\prime})  = 3 $ for every such $\g$, as follows from Theorem \ref{G-category of surfaces}.
\end{rema}

Observe, that for a surface $\Sigma_\g$ with an action  of $G$ such that $ r > n_{\g^{\prime}}$ Corollary  \ref{estimates of ct_g for surfaces} gives
$\ct_G(\Sigma_\g) \geq r = \cat_G(\Sigma_{\g^\prime})$, i.e., the rate of growth is at least linear in $r$.

\zz{However, we conjecture that if  $ r> n_{\g^\prime}$ then $\ct_G(\Sigma_\g)\geq \frac{\cat_G(\Sigma_\g) (\cat_G(\Sigma_\g))+1)}{2} $??}


\section{Estimate of $G$-covering type by $G$-genus}\label{Estimate of $G$-covering type by $G$-category and $G$-genus}

In this section we give an estimate from below of the equivariant covering type expressed in terms of the equivariant genus $\gamma_G$
analogous to the inequality of \cite[Thm 2.2]{GovcMarzantowiczPavesic2019} in which the covering type is estimated by
the Lusternik-Schirelmann category.

\begin{theorem}\label{thm:estimate by genus}
Let $X$ be  a $G$-space which has a structure of a $G$-complex (or, more generally, a $G$-CW complex). Assume that
the orbit types of the action of $G$ on $X$ are ordered linearly $(H_1) \geq (H_2) \geq \,\cdots\, \geq (H_k)$
(in particular, there is a unique  minimal orbit type). Then
$$ \ct_G (X) \,\geq\, \frac{1}{2} \, \gamma_G(X) \, (\gamma_G(X) +1)\,. $$
\end{theorem}

\begin{rema}
Note that the assumption of Theorem \ref{thm:estimate by genus} is satisfied if the action is free or with one orbit type.
It is also satisfied for every $G$-space $X$ if $G$ is a  group whose subgroups are linearly ordered, e.g., if $G= \bz_{p^k}$ for $p$ a prime.
Moreover, if we assume that the regular good $G$-cover is strictly regular then in the hypothesis of Lemma \ref{nonempty itersection implies cat_G=1}
the assumption about the order of orbit types is not necessary.
\end{rema}

Let $X$ be $G$-space with a structure of a $G$-complex and let $\UU$ be a good $G$-cover of $X$. We can partition $\UU$ into $G$-invariant
subsets  by taking orbits $\tilde U_i:= G  U_i$. With this notation in mind we have the following auxiliary result.

\begin{lemma}\label{nonempty itersection implies cat_G=1}
If  $\tilde{U}_1 \cap \,\cdots\, \cap \,\tilde{U}_{k+1} \neq \emptyset$ then
$\cat_G({\overset{k+1}{\underset{i=1}\bigcup}}\, \tilde{U}_i) = 1$ and, consequently, $\gamma_G({\overset{k+1}{\underset{i=1}\bigcup}}\, \tilde{U}_i)=1$.
\end{lemma}
\begin{proof}
Using the $G$-nerve theorem we can replace $Y= {\overset{k}{\underset{i=1}\bigcup}}\, \tilde{U}_i$ by  a $G$-homotopy equivalent  $G$-complex $L$.
Moreover, by replacing, if necessary, an open set $U_i \subset \tilde{U}_i= GU_i$ by another element $gU_i$ of the orbit, we can assume that
${U}_1 \cap \,\cdots\, \cap \,{U}_{k+1} \neq \emptyset$. Since $\{U_i\}_{i=1}^{k+1}$  form a good cover of $Y$,
${\overset{k+1}{\underset{i=1}\cup}}\, {U}_i$ is homotopy equivalent to a $k$-simplex $\sigma \subset L$ and $L=G \sigma= \tilde{\sigma}$ is
its orbit. Note that for a regular simplicial action the isotropy groups of points of faces of simplex are greater or equal to the isotropy groups
of points of interior of simplex. Consequently, the minimal isotropy groups appear at the vertices of orbit of  an simplex, therefore
$\tilde{\sigma}$ (the orbit of $\sigma$) can be described as the set
$$ \left\{[t_1 \,\tilde{g}_1  v_1,\, \dots,\, t_{k+1} \, \tilde{g}_{k+1} v_{k+1}]\ \bigg|  \; \; g\in G, \; \; \sum_{i=1}^{k+1}\, t_i=1\right\},$$
where $\tilde{g}_i \in G/H_i$ is the class of $g_i\in G$ in $G/H_i$  and $H_i=G_{v_i}$ is the isotropy group of vertex $v_i$. Under
this identification $\sigma$ corresponds to
$$  \left\{[t_1 \,\tilde{e}_1  v_1,\, \dots,\, t_{k+1} \, \tilde{e}_{k+1} v_{n+1}]\ \bigg| \;  \sum_{i=1}^{k+1}\, 
t_i=1 \right\} .$$
Let  $H=G_{v_{i}}$, for $1\leq i \leq k$,  be the minimal orbit type on $Y$, i.e. on $ L=\tilde{\sigma}$. Without loss of generality, we may assume
that $i=1$. By the assumption about linear ordering of orbit types on $X$, thus on $Y$, for every $1 \leq i\leq k$ there exists a $G$-map
$\phi_i: G/H_i \to G/H_1$ with $\phi_1=\mathrm{id}$. We can define a $G$-map $\phi: \tilde{\sigma} \to G/H_1 \times  \sigma $ by
$$\phi([t_1 \,\tilde{g}_1  v_1,\, \dots,\, t_{k} \, \tilde{g}_{n+1} v_{n+1}])=[ t_1\, \alpha_1(\tilde{g}_1)v_1\,,\dots\,, t_{n+1}\,
\alpha_{n+1}(\tilde{g}_{n+1})v_{n+1}].$$

By taking the composition of $\phi$ with the projection onto $G/H_1$, i.e.  $  p_1 \circ \phi: \tilde{\sigma} \to G/H_1$, and identifying $G/H_1$
with $Gv_1 \iota \tilde{\sigma}$   we get a $G$-equivariant map of $\psi = \iota \circ  p_1\circ \phi:\tilde{\sigma}  \to Gv_1$ onto the orbit.
This map is a $G$-deformation retraction by the deformation $\psi = \iota \circ  p_1\circ \phi_s$, where
$$\phi_s([t_1\,\tilde{g}_1 v_1 \, ,t_2\,\tilde{g}_2 v_2\,, \dots,\,t_{k+1}\, \tilde{g}_{k+1} v_{k+1}]) =
[t_1\,\tilde{g}_1 v_1 \, ,(1-s)t_2\,\tilde{g}_2 v_2\,, \dots,\,(1-s)t_{k+1}\, \tilde{g}_{k+1} v_{k+1}])\,,$$
which shows that $\cat_G(\tilde{\sigma}) = \cat_G({\underset{1}{\overset{n+1}{\cup}}}\, \tilde{V}_i)=1$.
\end{proof}

\begin{proof}[Proof of Theorem \ref{thm:estimate by genus}]
Let $X$ be a $G$-space whose strict $G$-covering type coincides with the $G$-covering type, let
$\mathcal{U}= \{\tilde{U}_s\} $, $1\leq s\leq \ct_G(X)$ be its good $G$-cover and denote by $K$ the nerve complex of $\UU$.
The vertices of $K$ correspond to elements of $\mathcal{U}$ which are split into $\ct_G(X)$ orbits, and $k$-faces correspond to
non-empty intersections $U_{i_1}\cap U_{i_2}\cap \, \cdots \,  \cap U_{i_k}$. Note that $\gamma_G(|K|)=\gamma_G(X)$ and $\ct_G(|K|)= \ct_G(X)$.

We proceed by induction on $\gamma_G(X)$. If $\gamma_G(X) =1 $ then the right hand side of inequality is equal to $1$, but $\ct_G(X)\geq  1 $,
and the inequality is clearly satisfied.

Assume that the estimate holds for spaces with $\gamma_G(X) \leq n$, and let $\gamma_G(X) = n+1$.
By our assumption on the order of orbit types and Proposition \ref{upper estimate of genus}, $\dim K= \dim (X/G)=\dim (X) \geq n$,  so that there exist sets $U_1,\ldots, U_{n+1} \in \mathcal{U} $ which intersect non-trivially. Since $\mathcal{U}$ is regular, the orbits  $\tilde{U}_s = G U_s$  are different for distinct $i$.
Let $\tilde{U}:= \tilde{U}_1\cup\cdots \cup\, \tilde{U}_{n+1}$ and let $\tilde{U}^\prime$ be the union of orbits of the remaining  elements of
$\mathcal{U}$.  By Lemma  \ref{nonempty itersection implies cat_G=1} we have $\gamma_G(\tilde{V})=1$. Hence $\gamma_G(\tilde{V}^\prime)\geq
\gamma_G(X)- \gamma_G(\tilde{V}) \geq n$ by the subadditivity of genus. If $\gamma_G(\tilde{V}^\prime) =n$ we can use
the induction assumption to compute
$$\ct_G(X) \geq   (n+1)+\frac{1}{2}\, n(n+1) = \frac{1}{2} \,(n+1)(n+2)\, .$$
If $\gamma_G(\tilde{V}^\prime) =n+1$ we repeat the same argument subtracting each time $n+1$ orbits $\tilde{U}_s$  from  the good $G$-cover until the
$G$-genus drops down by one and then  we can use the induction assumption.
\end{proof}

\begin{rema} A similar strategy can be used to obtain an estimate of the $G$-covering type using the $G$-category. However, since the statement of Proposition \ref{upper estimate of genus} does not hold in general for $\cat_G(X)$, the corresponding formula would be more complicated and
thus less useful.
\end{rema}

\begin{exem}
Let $X:=S(V)$ be the sphere bundle of an $(n+1)$-dimensional complex free representation $V$ of $G=\bz_p$.
Then $\gamma_G(S(V))= \dim_\br(V)= 2n+2$ (cf. \cite{Bartsch}) and  $\cat_G(S(V))= \dim_\br(V)= 2n+2$ (cf. \cite{Marzan}).
By Theorem  \ref{thm:estimate by genus} we get
$$\ct_G(S(V)) \geq (n+1) \, (2n+3)\,,$$
Since in this case $\ct_G(S(V)) = \ct(S(V)/G) = \ct(L^{2n+1}(p))$ we get the same as estimate  of  $\ct(L^{2n+1}(p))$  as the one given in
\cite{GovcMarzantowiczPavesic2022} (and stronger than a previous estimate of \cite{GovcMarzantowiczPavesic2019}).
\end{exem}

Observe that we may also reverse the above estimates to obtain upper bounds for the $G$-genus of a $G$-space based
on the cardinality of some  invariant regular good $G$-cover or on the number of vertices in an equivariant regular
triangulation of a $G$-manifold.

\subsection{\label{section3} Estimate of genus by the length index in the equivariant  $K$-theory}

In this subsection, we briefly recall the notion of an equivariant
index based on the  cohomology length in a  given cohomology theory.
It was introduced and described in details by Thomas Bartsch in
\cite[Chapter 4]{Bartsch} and it can be used to estimate from below the $G$-genus of a $G$-space.
We will use in for two generalized equivariant cohomology theories, the equivariant $K$-theory and the Borel cohomology.

First we consider the equivariant $K$-theory, denoted $K^*_G(X)$.
It is generated by the $K_G$-theory,  the equivariant
$K$-theory of $G$-vector bundles, and is extended to a graded cohomology
theory by the equivariant Bott periodicity (see Segal \cite{Segal}).

Let us fix a set  $\mathcal{A}$ of nontrivial  orbits, which is obviously finite, since $G$ is finite.

\begin{defi} The $(\mathcal{A},K_G^*)$ -- cup length of a pair $(X,
X^\prime)$ of $G$-spaces is the smallest $r$ such that there exist
$A_1, \,A_2, \,\ldots\, ,A_r \in \mathcal{A}$ and $G$-maps $\beta_i:
A_i \to X$, $1\leq i \leq r$ with the property that for all $\gamma
\in K^*_G(X,X^\prime)$ and for all $\omega_i \in \ker \beta_i^* $ we
have
\begin{eqnarray*}\omega_1\, \cup\omega_2\cup\,
\cdots \, \cup\omega_r \,\cup \, \gamma \,=\, 0\,\in K^*_G(X,
X^\prime).\end{eqnarray*}

\noindent If there is no such $r$, then we say that the
$(\mathcal{A},K_G^*)$ -- cup length of $(X,X^\prime)$ is $\infty$,
and $r=0$ means that $K^*_G(X, X^\prime)=0$. Moreover, the
$(\mathcal{A}, K_G^*)$ -- cup length of $X$ is by definition the cup
length of the pair $(X,\emptyset)$.\end{defi}

Recall that for the equivariant cohomology theory $K^*_G$ and a
$G$-pair $(X,X^\prime)$, the co\-ho\-mo\-lo\-gy $K^*_G(X, X^\prime)$
is a module over the coefficient ring $K^*_G({\rm pt})$, via the
natural $G$-map $p_X : X \to {\rm pt}$. We write
\begin{eqnarray*}\omega \cdot \gamma= p_X^*(\omega) \cup \gamma,
\,\, \text{and} \,\, \omega_1 \cdot \omega_2 = \omega_1 \cup
\omega_2,\end{eqnarray*} for $ \gamma \in K^*_G(X,X^\prime)$ and
$\omega_1, \, \omega_2 \in K^*_G({\rm pt})$.
By taking  $R:=K_{G}({\rm pt})=R(G)\subset K^*_G({\rm pt})$, we obtain
the following version of \cite[Definition 4.1]{Bartsch}.

\begin{defi} The $(\mathcal{A},K_G^*, R)$ -- {{\it length index}} of a pair $(X,
X^\prime)$ of $G$-spaces is the smallest $r$ such that there exist
$A_1, \,A_2, \,\ldots\, ,A_r \in \mathcal{A}$ with the following
property:

For all $\gamma\in K^*_G(X,X^\prime)$ and all $\omega_i \in
R\cap\ker (K^*_G({\rm pt})  \to K^*_G(A_i))=\ker ( K_{G}({\rm pt})
\to K_G(A_i))$, $i=1, \, 2,\dots, r$, the product
\begin{eqnarray*}\omega_1 \cdot\omega_2 \cdots \omega_r\cdot\gamma
=0\in K^*_G(X, X^\prime).\end{eqnarray*}
\end{defi}

A comparison between the length index and the cup-length  is given in the
following statement (cf. Bartsch \cite[p. 59]{Bartsch}).

\begin{propo}\label{length and index}
For any system $\mathcal{A}$ and  every pair of $G$-spaces $(X,
X^\prime)$ we have
\begin{eqnarray*}(\mathcal{A},K_G^*, R)\,-\,{\text{length index of }}\;
(X,X^\prime)\,\leq \,(\mathcal{A},K_G^*)\,-\,{\text{cup length of }}
\; (X,X^\prime).
\end{eqnarray*}
\end{propo}

The $(\mathcal{A},K_G^*, R)$ -- {\it{length index}} has many
properties which are important from the point of view of
applications to study critical points of $G$-invariant functions and
functionals (see \cite{Bartsch}), and we are going to use some of them.

For the rest of this subsection let $G=\bz_{pn}$. Following \cite{Bartsch}, given two powers $m,n$ of $p$ satisfying $ 1 \leq m\leq n\leq
p^{k-1}$  we set
\begin{equation}\label{A(m,n)}\mathcal{A}_{m,n}:= \{ G/H\, | \, H \subset G;\, m\leq \vert H\vert \leq n\,\},
\end{equation}
where $\vert H\vert $ is the cardinality of $H$. Next we put
\begin{equation}\label{l_n}
\ell_n(X, X^\prime) := (\mathcal{A}_{m,n},K_G^*,R)-{\text{\it length
index of }}\;\; (X, X^\prime)\,.
\end{equation}

\begin{rema}\label{Observation 5.5}\rm
By \cite[Observation 5.5]{Bartsch}, the index $\ell_n $ does not depend
on $m$. It says that if $\mathcal{A}^\prime \subseteq
\mathcal{A}$ is such that for each $A\in \mathcal{A}$ there exists
$A^\prime \in \mathcal{A}^\prime$ and a $G$-map $A\to A^\prime$, then
\begin{eqnarray*}
(\mathcal{A},K_G^{*},R)-{\text{length index}} \;=\;
(\mathcal{A}^\prime,K_G^{*},R)-{\text{length index}}.
\end{eqnarray*}
\end{rema}

For us important is the following.
\begin{prop}\label{prop: genus by length}(cf. \cite[Corollary 4.9]{Bartsch})
Given a $G$-space $X$ and a family of orbits $\mathcal{A}$   we have
$$\ell(X) \leq \gamma_G(X)\,,$$
where the length $\ell$ in a generalized equivariant cohomology $h_G^*$ and the $G$-genus are taken relative to  $\mathcal{A}$.
\end{prop}

We will write $\mathcal{A}_X$ for the set of all $G$-orbits of $X$ (up to isomorphism of finite $G$-sets).

\begin{theorem}(\cite[Theorem 5.8]{Bartsch})\label{thm:estimate of l_n}
Let $V$ be an orthogonal representation of $G=\bz_{p^k}$ satisfying
$V^G=\{0\}$, and let $d= \dim_\bc V =\frac{1}{2}\, \dim_\br V$. Fix $m,\,n $ two
powers of $p$ as above. Then
\begin{eqnarray*}
\ell_n (S(V))\, \geq \,
\begin{cases} 1 +
\Big[\frac{(d-1)m}{n}\Big]\;\;{\text{if}}\; \; \mathcal{A}_{S(V)}
\subset \mathcal{A}_{m,n}, \cr \infty \phantom{+
\Big[\frac{(d-1)m}{n}\Big]} \;\;{\text{if}} \;\;\mathcal{A}_{S(V)}
\nsubseteq \mathcal{A}_{1,n}\, ,
\end{cases}
\end{eqnarray*}
where $[x]$ denotes the least integer greater than or equal to $x$.
Moreover, if $ \mathcal{A}_{S(V)} \subset \mathcal{A}_{n,n}$, then
\begin{eqnarray*}\ell_n (S(V))= d\,.\end{eqnarray*}
\end{theorem}
As a direct consequence of Theorems \ref{thm:estimate by genus} and \ref{thm:estimate of l_n} we get the corresponding estimate of equivariant covering
type and the number of vertices of a regular invariant triangulation of the orthogonal actions on spheres.
\begin{theorem}\label{ct of spheres for Z^{p^k}}
Let $V$ be an orthogonal representation of $G=\bz_{p^k}$, and $m,\, n, \,, d  $ as in Theorem \ref{thm:estimate of l_n}. If $\mathcal{A}_{S(V)}
\subset \mathcal{A}_{m,n}$ then
$$ \ct_G(S(V)) \geq \frac{1}{2} \, \bigl(1 +
\Big[\frac{(d-1)m}{n}\Big]\bigr) \big(2 +
\Big[\frac{(d-1)m}{n}\Big]\bigr)$$
\end{theorem} \hfill $\square$

\begin{rema} Note that if $k\geq 2$ and $m\neq n$, then $S(V)$ from Theorem \ref{thm:estimate of l_n} has more than one orbit type, in particular,
it is not a free $G$-space.
\end{rema}

Using Theorem \ref{ct of spheres for Z^{p^k}} we estimate the $G$-covering type in a slightly more complicated situation.
\begin{prop}\label{estimate of ct for cyclic}
Let $G=\bz_m$ be the cyclic group with $m= p_1^{k_1}\, p_2^{k_2}\,\cdots \, p_r^{k_r}$, $p_i$ prime.  Choose
orthogonal representations $V_1,\ldots,V_r$ of $G$, such that, for each $i$, $V_i$ is induced by a representation of
 $\bz_{p_i^{k_i}}$. Assume furthermore that $V_i^G=\{0\}$ for all $i$.  Then
 $$\ct_G(S(V_1\oplus V_2\, \oplus \cdots \, \oplus V_r))\,=\,  \ct_{G_1}(S(V_1)) + \cdots + \ct_{G_r}(S(V_r))\,, $$
 where $G_i =\bz_{p^{k_i}}$ and  $\ct_{G_i}(S(V_i))$ is estimated in Theorem \ref{ct of spheres for Z^{p^k}}.
\end{prop}
\begin{proof}
We show the statement in the case $r=2$. The general case follows by the induction.
Let  $G= G_1\times G_2$ where $G_1=\bz_{p_1^{k_1}}$ and $ G_2 = \bz_{p_2^{k_2}}$. First observe that
$S(V)^{G_1}= S(V_2)$, $S(V)^{G_2}= S(V_1)$ and $ S(V)= S(V_1) * S(V_2) $, where the action of $g=(g_1,g_2)$ on the element $[t v_1, (1-t)v_2]$
is defined as $g [t\,  v_1, (1-t)\, v_2] = [t\, g_1 v_1, (1-t)\, g_2v_2]$.

Let $\mathcal{U}$ be a good $G$-cover of $ S(V)$ split into $\{\tilde{U}_j\}_{j=1}^n $, $ n= \ct_G(S(V))$ orbits.
If $\tilde{U}_j \cap S(V)^{G_i} \neq \emptyset $, $i=1,2$, then $\tilde{U}_j$ is $G$-contractible to an orbit in $S(V_i)=S(V)^{G_{i}}$ as
a $G$-map preserves fixed points of subgroups. By the same argument if  $\tilde{U}_j \cap S(V)^{G_i} \neq \emptyset $, then
$\tilde{U}_j \cap  S(V)^{G_i} = \emptyset $ as  $S(V)^{G_i}  \cap  S(V)^{G_i}= \emptyset $.  Moreover, the families
$\{\tilde{U}_j \cap S(V)^{G_i}\}_{j=1}^{n_i}$, $i=1,0$, $n_1+ n_2 \leq n$ form a good $G$-covers of $S(V)^{G_i}$ by
Definition \ref{defi:G-good cover II} and Proposition \ref{prop:comparison of Definitions}. This shows that
$\ct_G(S(V)) \geq \ct_G(S(V_1)) +\ct_G(S(V_2))$, since $\ct_G(S(V_i)) = \ct_G(S(V)^{G_{1-i}})) = \ct_{G_i}(S(V_i))$, $i=1,2$
and proves the inequality in one direction.

To show the opposite inequality we consider $G$-equivariant deformation retractions of the sets $W_1= \{ [t\, v_1 +(1-t) v_2]:\, t\neq 1 \} $ onto
$S(V_1) {\overset{G}\subset} S(V)$, and $W_2= \{ [t\, v_1 +(1-t) v_2]:\, t\neq 0 \}$ onto $S(V_2) {\overset{G}\subset} S(V)$,
denoted respectively as $\rho_1$ and  $\rho_2$ and given by
$\rho_1([t\, v_1 +(1-t) v_2])=  v_1]$, and $\rho_2([t\, v_1 +(1-t) v_2])=  v_2$.

Let $\mathcal{U}_i= \{\tilde{U}_s\}_{s=1}^{n_1}$, $i=1,2$,  be a good $G$-cover of $S(V_i)$ with $n_i = \ct_{G}(S(V_i))$. Then
the family of sets $ \mathcal{U}=\{\rho_1^{-1}(\mathcal{U}_1) \cup \rho_2^{-1}(\mathcal{U}_2)$ forms a good $G$-cover of $S(V)$ of
cardinality $n_1 +n_2$ as can be checked directly. This shows that $\ct_G(S(V)) \leq  \ct_G(S(V_1)) +\ct_G(S(V_2))$ and consequently
proves the statement.
\end{proof}

\begin{exem}\label{exem:action of cyclic with two points}
Let $W$ be a $d$-dimensional orthogonal representation  of the cyclic group $G=\bz_m$, such that the action of $G=\bz_m\subset S(\bc)$
as the roots of unity is free on $S(W)$. Note that $d$ can be arbitrary if $m=2$, otherwise $d$ must be even. Define $V:= W\oplus \br^1$;
then $S(V)= S(W)*S(\br)$ and  the action of $G$ on $S(V)$ is has exactly two fixed points  - the poles. Using the same construction as
in the proof of Theorem \ref{ct of spheres for Z^{p^k}} one can show that $$\ct_G(S(V) \leq \ct(S(\br)) +\ct_G(S(W)) = 2 + \ct_G(S(W))$$
If $\dim W = 2 $ this leads to the inequality $\ct_G(S(V))\leq 2 + 3 = 5$.

On the other hand,  if $d=2$, then $\dim_\br(S(W)) =1$, and consequently $\dim S(V)=2$.
Since this action on $S(V)=\bs^2$ preserves the orientation, we can apply Theorem \ref{thm:surface minimal invariant triangulation} and obtain
the equality $\ct_G(S(V))= \ct(S(V)/G) = \ct(\bs^2)= 4$.
\end{exem}

\section{Cohomological estimates} \label{sec:Cohomological estimates}

There are several lower estimates of the equivariant Lusternik-Schnirelmann category
$\cat_G(X)$ which are based on the multiplicative structure of  the cohomology ring $\tilde{H}^*(X/G;R)$,
or more generally of ${h_G}^*(X)$, where $h^*_G$ is any $G$-equivariant cohomology theory
(see Bartsch \cite{Bartsch}, Marzantowicz  \cite{Marzan}). This is due to the fact that cohomology products
detect intersection patterns of subspaces of $X$. In this section we will see that analogous ideas  play
an important role in the estimates of $G$-covering type and in the estimates of the number of orbits of
regular triangulations of $G$-spaces. These estimates, when available, can
considerably improve the category and genus estimates from the previous section.
Additionally, this approach can be used even when the action of $G$ has fixed points.

The singular cohomology theory has a natural filtration given as $H^{(s)}(X):= {\underset{k=0}{\overset{s}
{\oplus}}}\, H^k(X)$ which is related to the geometrical filtration of $X$ by its skeleta. In order to define
an analogous filtration for a
generalized equivariant cohomology theory $h^*_G$ we proceed as in
Segal \cite{Segal}.  Let $h_G^*$ be a generalized equivariant cohomology theory.
For a $G$–CW-complex $X$ we define a filtration of $h_G^*(X)$ as
$$h^*_{G,s} (X):=   \ker ( h_G^*(X)\to h^*_G(X^{(s-1)})).$$
This corresponds to the filtration of $X$ by $G$-subspaces $X^s=\pi^{-1}(Y^{(s)})$,  where $Y= X/G$ is
the orbit space whose cellular structure is induced by the $G$-CW-structure of $X$.  More generally, if $X$
is a $G$-space for which there exists a $G$-simplicial complex $K$ and a $G$-homotopy equivalence
$w: X\to |K|$,  then we take a filtration of $X$ by $G$-subsets $X^{(s)}= w^{-1}(|K^{(s)}|$ and
the corresponding filtration of $h^*_G(X)$ defined as above. The general case is discussed in
Segal \cite[Section 5]{Segal} using the nerve of a $G$–stable closed
finite covering of $X$, which is not necessarily a good cover.
We are at position to define the degrees of element of $h_G^*(X)$.

The filtration of $h_G^*(X)$ defined above  is decreasing ($d= \dim X$):
\begin{equation}\label{filtration}
h_G^*(X)=h_{G,0}^*(X) \supset h^*_{G,1}(X) \supset\cdots \supset h^*_{G,d-1}(X) \supset h^*_{G,d}(X) = 0
\end{equation}
and $h_G^*(X)$ is a filtered ring, because
\begin{equation}
 h_{G,s}^*(X) \, \cdot \, h_{G,s^\prime}^*(X)\,\subseteq h_{G,s+s^\prime}^*(X)
\end{equation}
As a consequence, each $h_{G,s}^*(X)$ is an ideal in $h_G^*(X)$. In particular,
we have the following characterization of $h_{G,1}^*(X)$ (cf. Segal
\cite[Proposition 5.1(i), page 146]{Segal}  )
\begin{equation}
h_{G,1}^*(X) = \ker \left(h^*_G(X)\to \prod_{x\in X} h_G^*(G/G_x)\right) = \bigcap_{x\in X} \ker\big( h_G^*(X)
\to h_G^*(G/G_x)\big)
\end{equation}

\begin{defi}\label{degree of element}
The \emph{degree} $|u|$ of an element $u\in h_G^*(X)$ is the maximal  $i$, such that  $u\in  h_{G,i}^*(X)$.
\end{defi}

Given an $n$-tuple of positive integers $i_1,\ldots,i_n \in \NN$ we will say that a $G$-space $X$
\emph{admits an essential $(i_1,\ldots,i_n)$-product} in $h_G^*$ if there are cohomology classes
$u_k \in h_{G,i_k}^*(X)$, such that the product \newline $u_1\cdot u_2\cdots u_n$ is non-trivial.
For every $(i_1,\ldots,i_n)$ there exist a $G$-space $X$ and a generalized equivariant cohomology theory
 $h_G^*$  such that admits an essential $(i_1,\ldots,i_n)$-product.
For example we can take $X=S^{i_1}\times\cdots\times S^{i_n}$, where $S^{i_n}= S(V_{i_n})$ is the sphere of
an  orthogonal orientation-preserving representation of $G$ and $h_G^*(X):= H^*(X; \bz)$.
Clearly, if $X$ admits an essential $(i_1,\ldots,i_n)$-product then so does every  $Y{\overset{G}{\simeq}} X$,
since their cohomology rings are isomorphic.

We may therefore define the $G$-\emph{covering type of the $n$-tuple of positive integers}
$(i_1,\ldots,i_n)$ with respect the equivariant cohomology theory  $h_G^*$ as
$$\ct_G(i_1,\ldots,i_n):=\min\big\{\ct_G(X)\mid X\,\text{admits\ an\ essential}\ (i_1,\ldots,i_n)
\mathrm{-product}\big\}$$
The following proposition follows immediately from the definition.

\begin{proposition}\label{prop:coho est}
$$ \ct_G(X) \geq \max \{\ct_G(|u_1|,\ldots,|u_n|)\mid \  \mathrm{for\  all}\ 0\neq u_1\cdots u_n\in h^*_G(X)\}$$
\end{proposition}

Although the covering type of a specific product of cohomology classes may appear as a coarse estimate, it will serve very well our purposes.
We will base our computations on the following technical lemmas. The first is a standard argument that we give here for the convenience of the reader.
To shorten the notation, we  denote $h_{G,1}^*(X)$ by $ \widetilde h^*_G(X)$ and note that our  definition is equivalent to the following
$$ \widetilde h^*_G(X) = {\underset{\phi: G {\overset{G}\to} X}{\bigcap}}\ker (\phi^*\colon  h_G^*(X) \to h_G^0(G))$$
\begin{lemma}\label{product of two elements 2}
Let $X=U\cup V$ where $U,V$ are open $G$-invariant subsets of $X$, and let $u,\,v\in\widetilde h^*_G(X)$ be cohomology classes whose product $u\cdot v$ is non-trivial.
If $U$ is $G$-categorical in $X$ then $i_V^*(v)$ is a non-trivial element of $h_G^*(V)$ (where $i_V$ denotes the inclusion map $i_V\colon V {\overset{G}{\hookrightarrow}} X$).
\end{lemma}
\begin{proof}
Assume by contradiction that $i_V^*(v)=0$. Exactness of the cohomology sequence
$$h_G^*(X,V) \stackrel{j_V^*}{\longrightarrow} h_G^*(X) \stackrel{i_V^*}{\longrightarrow} h_G^*(V)$$
implies that there is a class $\bar v\in h_G^*(X,V)$ such that $j_V^*(\bar v)=v$.
Moreover $i_U^*(u)=0$, because $i_U\colon U\hookrightarrow X$ is $G$-homotopic to the inclusion of an orbit  $i_{Gx} \subset X$, so there
is a class $\bar u\in h_G^*(X,U)$ such that $j_U^*(\bar u)=u$.
Then $u\cdot v=j_U^*(\bar u)\cdot j_V^*(\bar v)$ is by naturality equal to the image of $\bar u\cdot\bar v\in h_G^*(X,U\cup V)=0$, therefore $u\cdot v=0$, which
contradicts the assumptions of the lemma.
\end{proof}

By inductive application of the above lemma we obtain the following:

\begin{lemma}
\label{lem:subproducts}
Let  $u_1,\ldots,\,u_n\in\widetilde h_G^*(X)$ be cohomology classes whose product $ u_1\cdots u_n$ is
non-trivial, and let $X=U_1\cup\cdots\cup U_k\cup V$ where $U_1,\ldots,U_k$
are open, $G$-categorical subsets of $X$, and $V$ is an open $G$-invariant subspace of $X$. Then
the product of any $(n-k)$ different classes among $i^*_V(u_1),\ldots,i^*_V(u_n)$ is a non-trivial class
in  $h_G^*(V)$.
\end{lemma}

Toward the proof of the first statement, we begin the induction by using the following lemma.

\begin{lema}\label{lem:ct > i implies dim > i+2}
	For $ u \in h^*_G(X)$,
	if $ |u| \geq i$ then $\ct_G(X) \geq i+2$.
\end{lema}
\begin{proof}
Since $|u| \geq i$, we have  $\dim(X) = \dim(X/G)\geq i$.
From  Corollary \ref{coro:ct(X^*) < ct_G(X)} we get $\ct_G(i)\geq \ct(i) \ge i+2$.
To complete the proof we  make an  adaptation of   \cite[Proposition 3.1]{K-W}. Indeed, let $K$ be
a complex with simplicial regular action of $G$ which is  the nerve of a good regular $G$-cover
$\mathcal{U}$ of $X$.   By our assumption $K/G$ is a complex of dimension $\geq i$. Let
$\tilde{\mathcal{U}}^*$ be the induced good cover of the orbit space. If $\ct(|K|/G) = \ct(X/G)= i+1$, then
the complex $K/G$ consists of one $i$-dimensional simplex $\sigma^*$. This means that
$K =\pi^{-1}(\sigma^*) = G \sigma$ is the orbit of an $i$-dimensional simplex
$\sigma =[v_0,\,v_1,\,\dots\,,v_{i}] \subset K$. We denote $ G \sigma$ by $\tilde{\sigma}$. Since for every
subgroup $H\subset G$ the set $K^H$ is a sub-complex of $K \tilde{\sigma}$ and $H^\prime \subset H$ implies
$K^{H} \subset K^{H^\prime}$, the isotropy types of vertices are minimal (thus isotropy groups maximal!)
and for  any face of simplex $\tau \in K$ the isotropy types of each its face are smaller or equal to that
of $\tau$ as the action is given by
$$ g[t_0\,v_0, t_1\,v_1,  \dots\,, t_{i}\,v_{i_1}]= [t_0\,gv_0, t_1\,gv_1,  \dots\,, t_{i}\,gv_{i}]$$
and $G$ permutes vertices.

Let us form an invariant and closed cover of $K=\tilde{\sigma}$ by taking sets $\tilde{V}_j= G V_j$, where
$$V_j=\{[t_0\,v_0, t_1\,v_1,  \dots\,, t_{i}\,v_{i}]: \sum_{j=0}^{i} t_j =1, \; \; {\text{and}}\; t_j
\leq \frac{2}{3}\}\,.$$
Observe that $\tilde{V}_j$ is $G$-contractible to $\tilde{v}_j = G v_j$. Moreover, for every $k$-tuple,
$2\leq k \leq i$ the intersection $\cap \tilde{V}_j$ is $G$-contractible to the saturation of  barycenter
of the vertices corresponding to this tuple. In particular $\tilde{V}_1\cap \tilde{V}_2$ is $G$-contractible
to  the saturation of $[\frac{1}{2} v_1,\frac{1}{2} v_2, 0, \dots,0]$.
	
It shows that for every non-trivial intersection  $\tilde{h}_G^*(\bigcap \tilde{V}_j) = 0$.
The Mayer-Vietoris sequence argument gives $\tilde{h}^*_G(K)=0$ which contradicts $|u| \geq i$.
This shows that $\ct(|K|/G) \geq i+2$, and consequently $\ct_G(K)\geq i +2$
\end{proof}

We are ready to prove the main result of this section, an 'arithmetic' estimate for the covering type of
an $n$-tuple:

\begin{theorem}\label{thm:arithmetic}
$$ \ct_G(i_1,\ldots\,i_n)\geq  i_1 + 2\, i_2 + \, \cdots\, + n i_n + (n+1) $$
\end{theorem}
\begin{proof}
We will induct on the number of elements in the sequence $(i_1,\ldots\,i_n)$. The case $n=1$ is covered by
Lemma \ref{lem:ct > i implies dim > i+2}. Assume that the estimate holds for all sequences of length
$(n-1)$ and consider classes $u_1\in\widetilde h_{G,i_1}^*(X),\ldots,u_n\in\widetilde h_{G,i_n}^*(X)$
such that the product $u_1 \cdots u_n \in h_{G,i_1+\ldots +i_n}^*(X)$ is non-trivial.
By the equivariant Nerve theorem  \ref{thm:G-nerve homotopy equivalence} in every good $G$-cover $\UU$
there exists a $G$-invariant subset $\UU'{\overset{G}\subseteq} \UU$, containing at least
$(i_1+\ldots+i_n+1)$ orbits of sets that intersect non-trivially.
Let us denote by $U$ and $V$ respectively the unions of elements
of $\UU'$ and of $\UU-\UU'$. Again by
Nerve theorem, the set $U$ is $G$-contractible and thus $G$-categorical.
Then the restriction of $u_n$ on $U$ is trivial, so by Lemma \ref{lem:subproducts} the restriction of
$u_1\cdots u_{n-1}$ on $V$ is non-trivial. By the inductive assumption $\UU$ has at least
$$(i_1+2i_2+\ldots + (n-1)i_{n-1}+n)+(i_1+\ldots+i_n+1)=$$
$$= i_1+2 i_2+\ldots+n i_n + (n+1), $$
orbits of elements. 
\end{proof}

\begin{rema}
Note that the assumption of Theorem \ref{thm:arithmetic} does not require any condition on the orbit types
appearing in $X$.  This extends essentially applications of this theorem.
\end{rema}

As a first application of our results, let $P(V)$ be the projectivization of a complex $(n+1)$-dimensional
representation of $G$. The action  of $G$ on $V$ induces an action on $P(V)$, since
$g(\lambda \, v) = \lambda \, g(v)$ for $\lambda \in \bs^1 \subset \bc$.
The $G$-equivariant $K$-theory of $P(V)$ can be described as
$$K_G^0(P(V)) = \br(G)[\eta]/e(V)\,,$$
where $\br(G)$ is a representation ring of $G$ and $e(V)$ is an ideal in $\br(G)$  generated by
the element $ {\underset{i=0}{\overset{n}\sum}} \, (-1)^i \wedge^i (V) \, \eta^{n+1-i}$
(see Bartsch \cite{Bartsch} for details). Here $\eta$ denotes a $G$-vector bundle conjugated
to the $G$-Hopf bundle over $P(V)$. Furthermore $K_G^1(P(V)) = 0$.
By Theorem \ref{thm:arithmetic} we have the following estimate.

\begin{theorem}\label{thm:estimate for projective space}
Let $V$ be a complex representation of a finite group  $G$ of complex dimension $(n+1)$ and let
$P(V)$ be the projective space of $V$. Then
$$ \ct_G(P(V)) \,\geq (n+1)^2\,.$$
\end{theorem}
\begin{proof} First note that from the description of $K_G^0(P(V))$ it follows that
$ \eta^n= \eta \, \eta \, \cdots \,\eta \neq 0  $.
Next we have to show that the degree  $|\eta|$ is $2$.  A natural geometrical approach is to construct a
$G$-CW-complex structure on $P(V)$ in such a way that
$P(V_1)\subset P(V_1\oplus V_2)\, \subset \cdots\, \subset P(V)$, is an ascending filtration of $P(V)$ by
$G$ sub-complexes  given by the decomposition $V= V_1\oplus \, \cdots\, V_s$  into irreducible factors.
So for the $1$-$G$-skeleton we have  $P(V)^{(1)} \subset P(V_1)$. But then we have to do further analysis.

J. S{\l}omi{\'n}ska \cite{Slominska} constructed an equivariant version of the Chern character
$$ch_G: K^j_G(X) \to {\underset{i=0}{\overset{\infty} {\, \oplus\,}}} H^{2i+j}_{Br,G}(X;\mathcal{K})\,,\, j
= 0,\,1 \,,$$
where $H^{2i+j}_{Br,G}(X;\mathcal{K})$ is the singular equivariant Bredon cohomology with the system
$\mathcal{K}(G/H)= R(H)$ induced by the  $K_G$-theory.
By its construction,    $ H^{2i+j}_{Br,G}(X;\mathcal{K})=0$ for
$2i+j > \dim X$.
Moreover, in \cite{Slominska} it is shown that
$$ ch_G\otimes_\bq {\rm id}_\bq : K^j_G(X)\otimes \bq  \to {\underset{i=0}{\overset{\infty}
{\, \oplus\,}}} H^{2i+j}_{Br,G}(X;\mathcal{K}\otimes \bq )\,$$
is an isomorphism.

Let us show that $|e(V)| \geq 1 $. The restriction
$e(V)|_{P(V)^{(0)}}$ to any orbit $Gx_0=G/G_{x_0}$, $Gx_0 \in P(V)^{(0)}$  is given by the restriction of
$\eta$ to the fiber at $x_0$ (we view $V$ as a representation of $G_x$). In other words we may substitute
$\eta$ and the skew-symmetric powers of $V$ to the formula defining $e(V)$. This shows that
$e(V)|_{P(V)^{(0)}}= 0$, i.e. $|e(V)|\geq 1$.
Taking $X=P(V)^{(1)}$  we see that the image of  $K^*_G(P(V))\to  K_G^*(X)$ is the same as the image of
$ {\underset{i=0}{\overset{\infty} {\, \oplus\,}}} H^{2i+j}_{Br,G}(P(V);\mathcal{K}\otimes \bq )$ in
$ {\underset{i=0}{\overset{\infty} {\, \oplus\,}}} H^{2i+j}_{Br,G}(X;\mathcal{K}\otimes \bq )$.
But $ {\underset{i=0}{\overset{\infty} {\, \oplus\,}}} H^{2i+j}_{Br,G}(X;\mathcal{K}\otimes \bq )
= H^{0}_{Br,G}(X;\mathcal{K}\otimes \bq )$, since $ \dim X =1$.
On the other hand, $ch_G(\eta) = 0$ in  $H^{0}_{Br,G}(X;\mathcal{K}\otimes \bq )$ as we showed above.
Consequently $e(V)|_{P(V)^{(1)}} = 0$ which means that $|e(V)|\geq 2$.

By Theorem \ref{thm:arithmetic} we get $\ct_G(P(V)) \geq (n+1)^2 \,.$
\end{proof}

The topological dimension $\dim P(V) $ is equal to $d= 2n$, so we can express the covering type of $P(V)$
in terms of its geometric dimension as  $\ct_G(P(V)) \geq \frac{(d+2)^2}{4}$.

Another class of examples to which our methods could be applied are $\bz_p^k$-actions on
$F_p$-cohomology spheres and on spheres with different differential structures.
In this situation the most appropriate generalized cohomology theory is the Borel cohomology.

We begin with the following technical lemma.

\begin{lemma}\label{lem:relative ct}
Let $X$ be a compact $G$-CW-complex, such that $X^G \neq \emptyset$ and assume that $X$ has a finite regular
good $G$-cover $\mathcal{U} =  \{\tilde{U}_i\}, \, i\in I$ of cardinality $|I|= \ct_G(X)$. Furthermore, let
$\mathcal{U}^G$ be the subfamily of this cover consisting of $\{\tilde{U}_{i_s}\}, \, s\in S\subset I$ such
that $\tilde{U}_{i_s}\cap X^G \neq \emptyset $.

Then the orbit of every element of this subfamily $\tilde{U}_{i_s} \in \mathcal{U}^G $ consists of one
element $U_{i_s} = G U_{i_s}$, i.e. $U_{i_s}$ is $G$-invariant.  Moreover the cardinality
$|S|= \ct_G(X^G) = \ct(X^G)$.
Consequently the cardinality of the set of orbits of  family $\mathcal{U} \setminus \mathcal{U}^G$ is equal
to $ |I|-|S|= |J|$, where $J=I\setminus S$.

Furthermore $(X,X^G)$ is $G$-homotopy equivalent to a pair of  regular $G$-complexes $(K, K^G)$ with
the regular $G$-cover by the orbits of open  stars of vertices  $\mathcal{V}= \{\tilde{V}_i\}, \, i\in I$
and $\mathcal{V}^G= \{\tilde{V}_{i_s}\}, \, s\in S$, with $|I|=\ct_G(K)$,$|S|\geq \ct(K^G)$,   for which
$$ {\underset{j\in J}{\,\bigcup\,}} \tilde{V}_j \,=\, K\setminus K^G \,.$$
\end{lemma}
\begin{proof}  The fact that each member of the subfamily $ {\underset{j\in I\setminus S}{\,\bigcup\,}}
\tilde{V}_j \,=\, K\setminus K^G \,$ of  $ \mathcal{U}^G $ is $G$-invariant follows from condition RC1)
of Definition \ref{defi:regular cover} of the regular cover.

The $G$-homotopy equivalence of the last part of statement follows from  the $G$-nerve theorem
(Theorem \ref{thm:G-nerve homotopy equivalence}), which also gives a one-to-one correspondence between
$\mathcal{U} =  \{\tilde{U}_i\}, \, i\in I$  and  orbits of vertices of  $K$, thus  orbits of the good
regular $G$-cover $\mathcal{V} =  \{\tilde{V}_i\}, \, i\in I$ given by stars of vertices, and between
$ \mathcal{U}^G$ and $\mathcal{V}^G$ respectively. From the construction of $\mathcal{V}$ it follows that
$ {\underset{j\in I\setminus S}{\,\bigcup\,}} \tilde{V}_j \,=\, K\setminus K^G \,.$

To show that $|S|\geq \ct_G(X^G) = \ct(X^G)$ we  use  the $G$-cover of the already proved part of statement.
Since   the sets $ \{\tilde{U}_{i_s}\}_{s\in S} \cap |K^G|$ form a good cover of $|K^G|$, we have
 $|S| \geq \ct(|K|^G)$.
\end{proof}

To study the equivariant covering type of $G$-spaces with a non-empty fixed point set we  need the relative,
modulo $X^G$  versions of $\ct_G$ and  essential product. As before, we assume that
$X$ has a structure of a $G$-CW-complex.

\begin{defi}\label{def:relative ct} For a pair $ X^G \subset X $, $X$ a $G$-CW complex we define the relative
strict $G$-covering type of the pair $(X, X^G)$ as
$$ \sct_G(X, X^G) \,=  \sct(X\setminus X^G) \,.$$
\end{defi}

First note that $\sct_G(X,X^G)$ is finite and $\sct_G(X,^G) +\ct(X^G) \leq \ct_G(X)$ as follows from Lemma
\ref{lem:relative ct}.
We conjecture that
\begin{conj}\label{conjecture on relative ct}
$$\sct_G(X,X^G) +\ct(X^G) =  \ct_G(X)$$
\end{conj}

\begin{defi}\label{def:relative strict G covering type}
Given an $n$-tuple of positive integers $i_1,\ldots,i_n \in \NN$ we say that a $G$-space $X$ with
$X^G\neq \emptyset$ \emph{admits an essential relative $(i_1,\ldots,i_n)$-product} in $h_G^*$
if there are cohomology classes $u_k \in h_{G,i_k}^*(X, X^G)$, such that the product  
$u_1\cdots u_n$ is non-trivial.
\end{defi}

Clearly, if $X$ admits an essential relative $(i_1,\ldots,i_n)$-product then so does every  space
$Y$ that is $G$-homotopy equivalent to $X$, since their cohomology rings are isomorphic.

The following proposition follows immediately from the definition.

\begin{proposition}\label{prop:coho est 2}
$$ \sct_G(X, X^G) \geq \max \{\ct_G(|u_1|,\ldots,|u_n|)\mid  0\neq u_1\cdots u_n
\in h^*_G(X,X^G)\}$$
\end{proposition}

From now on, we proceed as in the proof of first part of statement modifying  each argument to the relative
case. Note that $h_{G,1}^*(X,X^G) = \widetilde h^*_G(X,X^G) = h^*_G(X,X^G)$ as ${\rm pt}
{\overset{G}{\,\subset\,}} X^G {\overset{G}{\,\subset\,}} X$.

\begin{lemma}\label{product of two elements}
Let $X\setminus X^G=U\cup V$ where $U,V$ are open and $G$-invariant in $X$, and let $u,\,u\in h^*_G(X,X^G)$
be cohomology classes whose product $u\cdot v$ is non-trivial in $h^*_G(X,X^G)$
If $U$ is $G$-categorical in $X\setminus X^G$ then $i_V^*(u)$ is a non-trivial element of $h_G^*(V)$
(here $i_V$ stands for the inclusion map $i_V: V {\overset{G}{\hookrightarrow}} X$).
\end{lemma}
\begin{proof}
Assume by contradiction that $i_V^*(u)=0$. Exactness of the cohomology sequence
$$h_G^*(X,X^G\cup V) \stackrel{j_V^*}{\longrightarrow} h_G^*(X,X^G) \stackrel{i_V^*}{\longrightarrow} h_G^*(V\cup X^G, X^G)$$
implies that there is a class $\bar u\in h_G^*(X,V\cup X^G)$ such that $j_V^*(\bar u)=u$.
Moreover $i_U^*(v)=0$, because $i_U\colon U\hookrightarrow X$ is $G$-homotopic to the inclusion of an orbit  $i_{Gx} \subset X$, so there
is a class $\bar v\in h_G^*(X,U\cup X^G)$ such that $j_U^*(\bar v)=v$.
Then $u\cdot v=j_V^*(\bar u)\cdot j_U^*(\bar v)$ is by naturality equal to the image of $\bar u\cdot\bar v\in h_G^*(X,U\cup V)=0$, therefore $u\cdot v=0$, which
contradicts the assumptions of the lemma.
\end{proof}

By inductive application of the above lemma we obtain the following:

\begin{lemma}
\label{lem:relative subproducts}
Let  $u_1,\,\cdots,\, u_n\in  h_G^*(X, X^G)$ be cohomology classes whose product $ u_1\cdots u_n$ is non-trivial, and let $X=U_1\cup\ldots\cup U_k\cup V$ where $U_1,\ldots,U_k$
are open, $G$-categorical subsets of $X$, and $V$ is open invariant  in $X$. Then the product of any $(n-k)$ different classes among $i^*_V(u_1),\ldots,i^*_V(u_n)$ is a non-trivial class in  $h_G^*(V)$.
\end{lemma}

The above lemma implies the following relative version of Theorem \ref{thm:arithmetic}:
\begin{theorem}\label{thm: relative arithmetic}
Let  $u_1,\,\ldots,\, u_n\in  h_G^*(X, X^G)$ be cohomology classes whose product $ u_1\cdots u_n$ is non-
trivial. Then
$$ \sct_G(X,X^G)\geq  i_1 + 2\, i_2 + \, \cdots\, + n i_n + (n+1) $$
\end{theorem}

We are now ready to formulate a theorem which estimates $\ct_G$ or relative $\sct_G$ for an action of
$p$-tori on a $F_p$-cohomology sphere.

\begin{theorem}\label{thm:ct of spheres for Z_p^k}
Let $\bs^n$ be an $n$-dimensional manifold, homotopy equivalent to an  $F_p$-cohomology sphere, with an
action ot the group $G=\bz_p^k$, with $p$-prime and $k\geq 1$. Assume first that  $\bs^G=\emptyset$.
Depending on $p$ we have
$$ \begin{matrix}
\ct_G(\bs^n) \geq \frac{(n+1)(n+2)}{2} \;\; {\text{if}}\;\; p=2\,,\cr{\phantom{\text{second row}}}\cr
\ct_G(\bs^n) \geq \frac{(d)(d+1)}{2}\;\; {\text{if}}\;\; p>2, \;\;{\text{where}}\;\;d=\frac{n+1}{2}\;\;
\end{matrix}
$$
If $S^G \neq \emptyset$ then $S^G{\underset{F_p}{\,\sim\,}} \bs^r $ is a $F_p$ cohomology sphere of dimension
$r\geq 0$ and
$$ \begin{matrix}
\sct_G(\bs^n) \geq (r+2)+ \frac{(n-r-1)(n-r+2)}{2} \;\; {\text{if}}\;\; p=2\,, \cr{\phantom{\text{second row}}}
\cr
\ct_G(\bs^n) \geq (r+2) + \frac{(d-1) (d+1)}{2}\;\; {\text{if}}\;\; p>2, \;\;{\text{where}}\;\;d=\frac{n-r}{2}
\,.
\end{matrix}
$$
\end{theorem}

\begin{proof}
The proof is based on the Borel theorem (see \cite[Capt. IV]{Hsiang} for  an exposition). But for us a more
convenient source of material is \cite{Mat-San-Sil} where an approach to the Borel theorem given by
T. tom Dieck in \cite{Dieck} is  adapted to the problems discussed here.

The first part of the Borel theorem, which recovers the Smith theorem,  states that for any action of the group
$G=\bz_p^k$, or even more general for a finite $p$-group, we have that the fixed point set
$ \bs^G$ is $F_p$-cohomology equivalent to $\bs^r$, where $-1\leq r \leq n$, i.e.,  it is an $F_p$-cohomology
sphere of dimension $r\geq 0$, or it is empty if $r=-1$.

Let $\mathcal{H}+\mathcal{H}(\bs)$, denote the family  $\{H_1, H_2, \, \dots\,, H_t\}$, $1\leq t\leq p^{k-1}$
of sub-tori $H \subset G$  of rank $k-1$, such that $\bs^{H_i}\neq \emptyset$, thus $\bs^{H_i}\sim \bs^{n_i}$,
$n_i\geq 0$. The second part of Borel theorem states that
$$  n-r\,= \sum_{i=1}^t \, (n_i -r) \,.$$

Moreover, the theorem gives an information about the cohomology product in the Borel cohomology defined as
$H^*_G(\bs;F_p) = H^*(EG{\underset{G}\times} X;F_p)$, where the latter is understood as the limit over
filtration of the classifying space $EG$ by its skeletons.

To show the first part of theorem, assume that $\bs^G=\emptyset$, i.e. $r=-1$
Furthermore,  the Borel theorem says that for an action  with $\bs^G = \emptyset$ we have
$ H^*_G(\bs;F_p) = H^*(BG)/e(\bs)$ and $e(\bs)$ has the following description (cf. \cite[Theorem 3.1 and
Remark 3.2]{Mat-San-Sil}).  Namely, there exist elements $w_1,\,w_2, \, \dots\,, w_t \in H^*_G(\bs;F_2)$
such that $ e(\bs) =  w_1^{k_1} \cdot w_2^{k_2}\cdots w_t^{k_t} $ if $p=2$,  with
$k_i = n_i-r$, or correspondingly  elements  $ s_1,\,s_2, \, \dots\,, s_t \in H^*_G(\bs;F_p)$  such that
$ e(\bs)= s_1^{k_1} \cdot s_2^{k_2}\cdots s_t^{k_t} \neq 0$ if $p>2$,  with
$k_i = \frac{n_i-r}{2}$ respectively.
This means that the length of non-zero multiple of elements $w_1,\,w_2, \, \dots\,, w_t$, or
$s_1,\,s_2, \, \dots\,, s_t$ correspondingly is equal to $n+1 -1 = n$, and
$\frac{n+1}{2} -1 = \frac{n-1}{2}$ respectively, since $r=-1$.

In the case $r\geq 0$, i.e. if $(\bs^n)^G\neq \emptyset$ we have similarly
$ H^*_G(\bs, \bs^G;F_p) = H^*(BG)/e(\bs)$ with
$ e(\bs) =  w_1^{k_1} \cdot w_2^{k_2}\cdots w_t^{k_t}\neq 0 $ if $p=2$, but
$ e(\bs)= s_1^{k_1} \cdot s_2^{k_2} \cdots s_t^{k_t} + \mathfrak{n} \neq 0$ if $p>2$,
where $\mathfrak{n}$ is nilpotent (cf. \cite[Theorem 3.1]{Mat-San-Sil}).  This means that the length of
non-zero    multiple of elements $w_1,\,w_2, \, \dots\,, w_t$, or $s_1,\,s_2, \, \dots\,, s_t$
correspondingly is equal to $n-r -1 $, or $\frac{n-r}{2} - 1 = \frac{n-r-2}{2}$ respectively.

In order to apply Theorem \ref{thm:arithmetic} we have to show that the degree  $|w_i| \geq 1$, and
$|s_i| \geq 2$ respectively. But, it follows from the description of elements $w_1,\,w_2, \, \dots\,, w_t$,
and  $s_1,\,s_2, \, \dots\,, s_t$ respectively. Remind that $H_G^*(*; F_p)= H^*(BG;F_p)$ and
$H^*(B\bz_2^k; F_2)= F_2[\omega_1, \, \omega_2,\dots\,, \omega_k]$, is a polynomial ring,  and
$H^*(B\bz_p^k; F_p)= F_p[\gamma_1, \, \gamma_2\dots\,, \gamma_k] \otimes \wedge(\alpha_1\cdot\alpha_2\,
\cdots\, \cdot \alpha_k)$ is the tensor product of polynomial ring $R$ and the skew-symmetric algebra
respectively. Furthermore $w_i \in \ker (H^1(BG;F_p): {\overset{\iota^*}{\,\to\,}} H^1(BH_i; F_p)) $ induced
by the map $\iota:BH_i \to BG$, and in the case of  $p$-odd  $s_i \in  R\cup \ker (H^2(BG;F_p):
{\overset{p^*}{\,\to\,}} H^2(BH_i; F_p)) $. In other words $w_1,\,w_2, \, \dots\,, w_t$ and
$s_1,\,s_2, \, \dots\,, s_t$ are images of elements described above by the homomorphism
$p^*: H_G^*(X;F_p) \to  H^*_G(pt) = H^*(BG;F_p)$ induced by the $G$-map $p: EG\times X \to EG$ on the
orbit spaces. But in the singular cohomology theory  $H^*(BG; F_p)$ the degree of an element is equal to
its gradation, which shows that $\omega_i$ are of degree $1$, and $\gamma_i$ of degree $2$. Consequently
their images and their linear combinations are of the same degree which shows that $|w_i| =1 $, and
respectively  $|s_i|=2$.

To complete the proof in the case where $\bs^G=\emptyset$,  it is enough to apply Theorem
\ref{thm:arithmetic} with the rank of summation from $1$ to $n$ and with the degree of each factor equal to
$1$ if $p=2$,  or with the rank of summation from $1$ to $\frac{n-1}{2}$ and degree of each factor equal
to $2$ if $p>2$.

In the case when $\bs^G \neq \emptyset$,  we have
$\ct_G(X) \geq \ct(X^G) + \sct_G(X,X^G) = \ct(X^G) + \sct_G(X\setminus X^G)$ by   Lemma \ref{lem:relative ct}.
Since $\bs^G {\underset{F_p}{\,\sim \,}} \bs^r$ is the $F_p$ cohomology sphere of dimension $r$, we have
$\ct(\bs^G) = r+2$. To estimate the  second term we use Theorem \ref{thm: relative arithmetic} and the above.
\end{proof}


\begin{thebibliography}{99}

\zz{\bibitem{Bag-Dat} B.~Bagchi, B.~Datta,
\emph{Minimal triangulations of sphere bundles over the circle}, J. Combin. Theory Ser. A 115
(2008), no. 5, 737--752.}

\bibitem{Bartsch}
T. Bartsch,  Topological methods for variational problems
with symmetries, Lecture Notes in Mathematics, 1560, Springer-Verlag, Berlin, 199.

\bibitem{Barnette}
D.~Barnette, \emph{Graph theorems for manifolds}, Israel J. Math. 16 (1973), 62-72.

\bibitem{Bor-Min}
E.~Borghini, E.~ Minian,
\emph{The covering type of closed surfaces and minimal triangulations},
J. Combin. Theory Ser. A 166 (2019), 1–10.

\bibitem{Bredon}
G. Bredon, \emph{Introduction to Compact Transformation Groups}, Academic Press, New York, London, 1972.

\zz{\bibitem{Bre-Kuh}
U.~Brehm,  W.~K{\"u}hnel, \emph{Combinatorial manifolds with few vertices},
Topology 26, (1987), 465--473.}

\bibitem{Broughton}
S.~A.~ Broughton, \emph{Classifying finite group actions on surfaces of
low genus},  Journal of Pure and Applied Algebra 69, (1990), 233--270.

\bibitem{Clapp-Puppe}
M.~Clapp, D.~Puppe, \emph{Critical point theory with symmetries},
J. Reine Angew. Math. 418 (1991), 1--29.

\bibitem{Colman}
H.~Colman, \emph{Equivariant LS-category for finite group actions},
Lusternik-Schnirelmann category and related topics (South Hadley, MA, 2001), 35-40, Contemp. Math.
316, (Amer. Math. Soc., 2002).

\bibitem{CLOT}
O.~Cornea, G.~Lupton, J.~Oprea, D.~Tanre, \emph{Lusternik-Schnirelmann category}, Mathematical
Surveys and Monograph   hs, vol. 103, (American Mathematical Society, 2008).

\bibitem{Datta}
B.~Datta, \emph{Minimal triangulations of manifolds}, arXiv:math/0701735.

\bibitem{Dieck}
T.~tom Dieck, \emph{Transformation groups}, De Gruyter Studies in Mathematics, 8. Walter
de Gruyter \& Co., Berlin, 1987. x+312 pp.

\bibitem{DMZ}
H.~Duan, W.~Marzantowicz, and X.~Zhao, \emph{On the number of simplices required to
triangulate a Lie group}, Top. Appl., 2021, Paper No. 107559,

\bibitem{Fadell}
E.~Fadell, \emph{The relationship between Ljusternik-Schnirelman category and the
concept of genus},  Pacific J. Math. 89 (1980), no. 1, 33-42.

\bibitem{GovcMarzantowiczPavesic2019}
D.~Govc, W.~Marzantowicz, P.~Pave\v{s}i\'{c}, \emph{Estimates
of covering type and the number of vertices of minimal triangulations},  Discrete Comput. Geom. 63
(2020), no. 1, 31-48.

\bibitem{GMP2}
D.~Govc, W.~Marzantowicz and P.~Pave\v si\'c,
\emph{How many simplices are needed to triangulate a Grassmannian?},  TMNA,  (2020), {\bf 56}
no. 2, 501--518.

\bibitem{GovcMarzantowiczPavesic2022}
D.~Govc, W.~Marzantowicz, P.~Pave\v{s}i\'{c},
\emph{Estimates of covering type and minimal triangulations based on category weight},
Forum Math. 34 (2022), no. 4, 969–988.

\zz{\emph{}\bibitem{GroMar}
G.~Gromadzki, W.~Marzantowicz, \emph{Conformal actions with prescribed periods on Riemann surfaces},
Fund. Math. 213 (2011), no. 2, 169–190.}

\bibitem{GroJezMar}
G.~Gromadzki, J.~Jezierski, W.~Marzantowicz, \emph{Critical points of invariant
functions on closed orientable surfaces}, Bol. Soc. Mat. Mex. (3) 21 (2015), no. 1, 71–88.

\bibitem{Hatcher}
A.~Hatcher, \emph{Algebraic Topology}, (Cambridge University Press, 2002).

\bibitem{Hsiang}
W-Y.~Hsiang, Cohomology theory of topological transformation groups. Ergebnisse der Mathematik
und ihrer Grenzgebiete, Band 85. Springer-Verlag, New York-Heidelberg, 1975. x+164 pp.

\bibitem{Illman}
S.~Illman, \emph{Smooth equivariant triangulations of $G$-manifolds for $G$
a finite group}, Math. Ann. 233 (1978), 199--220.

\bibitem{Illman2} S.~Illman, \emph{The equivariant triangulation theorem for actions of compact Lie groups}, 
Math. Ann. 262 (1983), no. 4, 487--501.

\bibitem{K-W}
M.~Karoubi, C.~Weibel, \emph{On the covering type of a space},
L'Enseignement Math. 62 (2016), no. 3-4, 457–474.

\bibitem{Lutz}
F.~Lutz, \emph{Triangulated Manifolds with Few Vertices: Combinatorial Manifolds},
arXiv:math/0506372.

\bibitem{Lutz2}
F.~Lutz, \emph{Triangulated manifolds with few vertices and vertex-transitive group actions}.
(German summary) Dissertation, Technischen Universität Berlin, Berlin, 1999. Berichte aus der
Mathematik. [Reports from Mathematics] Verlag Shaker, Aachen, 1999. vi+137 pp.

\bibitem{Marzan}
W. Marzantowicz, A $G$-Lusternik-Schnirelman category of space
with an action of a compact Lie group, Topology 28 (1989), 403-412.

\bibitem{Mat-San-Sil}
D. de Mattos,  E. Lopes dos Santos, N. A.  Silva, \emph{On the length of cohomology spheres},
Topology Appl. 293 (2021), Paper No. 107569, 11 pp.

\bibitem{Mich-Viz}
P.~W.~Michor, C.~Vizman, \emph{$ n$ -transitivity of certain diffeomorphism groups}, Acta
Math. Univ. Comenianae (N.S.) 63 (1994), no. 2, 221-225.

\bibitem{Segal}
G.~Segal, \emph{Equivariant K–theory}, Inst. Hautes Études Sci. Publ. Math. (1968) 129-151.

\zz{\bibitem{Sierakowski}
M.~Sierakowski, \emph{Sets of periods for automorphisms of compact Riemann surfaces},
J. Pure Appl. Algebra 208 (2007), no. 2, 561–574.}

\bibitem{Slominska}
J.~S{\l}omi{\'n}ska, \emph{On the equivariant Chern homomorphism},
Bull. Acad. Polon. Sci. Sér. Sci. Math. Astronom. Phys. 24 (1976), no. 10, 909--913.

\zz{\bibitem{Tancer}
M.~Tancer, \emph{Intersection patterns of convex sets via simplicial complexes, a survey},
arXiv 1102.0417v2}

\bibitem{Tancer2}
M.~Tancer, D.~Tonkonog, \emph{Nerves of good covers are algorithmically unrecognizable}.
(English summary) SIAM J. Comput. 42 (2013), no. 4, 1697–1719.

\zz{\bibitem{Whitehead}
G.~W.~Whitehead, \emph{Elements of Homotopy Theory}, Graduate Texts in Mathematics, vol. 61,
(Springer, Berlin, 1978).}

\bibitem{Yang}
H.~Yang,  \emph{Equivariant cohomology and sheaves},  J. Algebra, 412 (2014), 230--254.
\end{thebibliography}
\end{document}